\documentclass[11pt]{preprint}
\usepackage{difftrees} %Added%
\usepackage[full]{textcomp}
\usepackage[osf]{newtxtext} 
\usepackage[cal=boondoxo]{mathalfa}
\usepackage{colortbl}

\usepackage{tikz-cd} %Added%

\usetikzlibrary{cd} %Added%
\usepackage[all,cmtip]{xy}
\usepackage{comment}

\usepackage{amssymb}
\usepackage{mathtools}
\usepackage{hyperref}
\usepackage{breakurl}
\usepackage{mhenvs}
\usepackage{mhequ} 
\usepackage{mhsymb}
\usepackage{booktabs}
\usepackage{tikz}
\usepackage{tcolorbox}
\usepackage{mathrsfs}
\usepackage[utf8]{inputenc}
\usepackage{longtable}
\usepackage{wrapfig}
\usepackage{rotating} %Added%
\usepackage{subcaption}
\usepackage{mathrsfs}
\usepackage{epsfig}
\usepackage{microtype}
\usepackage{comment}
\usepackage{wasysym}
\usepackage{centernot}
\usepackage{enumitem}
\usepackage{bm}
\usepackage{stackrel}

\usepackage{graphicx}
\usepackage{axodraw}
\usepackage{xspace}
\usepackage{subcaption} %Added%
\usepackage{epsfig} %Added%
\usepackage{axodraw} %Added%
\usepackage{xspace} %Added%
\usepackage[toc,page]{appendix} %Added%
\usepackage[all,cmtip]{xy} %Added%
\usepackage{relsize} %Added%
\usepackage{shuffle} %Added%
\makeatletter
\newcommand{\globalcolor}[1]{%
  \color{#1}\global\let\default@color\current@color
}
\makeatother

\usetikzlibrary{calc}
\usetikzlibrary{decorations}
\usetikzlibrary{positioning}
\usetikzlibrary{shapes}
\usetikzlibrary{external}

\definecolor{blush}{rgb}{0.87, 0.36, 0.51}
	\definecolor{brightcerulean}{rgb}{0.11, 0.67, 0.84}
	\definecolor{greenryb}{rgb}{0.4, 0.69, 0.2}

\newif\ifdark
\darkfalse

\ifdark
\definecolor{darkred}{rgb}{0.9,0.2,0.2}
\definecolor{darkblue}{rgb}{0.7,0.3,1}
\definecolor{darkgreen}{rgb}{0.1,0.9,0.1}
\definecolor{franck}{rgb}{0,0.8,1}
\definecolor{pagebackground}{rgb}{.15,.21,.18}
\definecolor{pageforeground}{rgb}{.84,.84,.85}
\pagecolor{pagebackground}
\AtBeginDocument{\globalcolor{pageforeground}}
\definecolor{symbols}{rgb}{0,0.7,1}
\colorlet{connection}{red!80!black}
\colorlet{boxcolor}{blue!50}

\else

\definecolor{darkred}{rgb}{0.7,0.1,0.1}
\definecolor{darkblue}{rgb}{0.4,0.1,0.8}
\definecolor{darkgreen}{rgb}{0.1,0.7,0.1}
\definecolor{franck}{rgb}{0,0,1}
\definecolor{pagebackground}{rgb}{1,1,1}
\definecolor{pageforeground}{rgb}{0,0,0}
\colorlet{symbols}{blue!90!black}
\colorlet{connection}{red!30!black}
\colorlet{boxcolor}{blue!50!black}

\fi

\def\slash{\leavevmode\unskip\kern0.18em/\penalty\exhyphenpenalty\kern0.18em}
\def\dash{\leavevmode\unskip\kern0.18em--\penalty\exhyphenpenalty\kern0.18em}

\DeclareMathAlphabet{\mathbbm}{U}{bbm}{m}{n}

\DeclareFontFamily{U}{BOONDOX-calo}{\skewchar\font=45 }
\DeclareFontShape{U}{BOONDOX-calo}{m}{n}{
  <-> s*[1.05] BOONDOX-r-calo}{}
\DeclareFontShape{U}{BOONDOX-calo}{b}{n}{
  <-> s*[1.05] BOONDOX-b-calo}{}
\DeclareMathAlphabet{\mcb}{U}{BOONDOX-calo}{m}{n}
\SetMathAlphabet{\mcb}{bold}{U}{BOONDOX-calo}{b}{n}
\setlist{noitemsep,topsep=4pt,leftmargin=1.5em}

\DeclareMathAlphabet{\mathbbm}{U}{bbm}{m}{n}

\DeclareMathAlphabet{\mcb}{U}{BOONDOX-calo}{m}{n}
\SetMathAlphabet{\mcb}{bold}{U}{BOONDOX-calo}{b}{n}
\DeclareFontFamily{U}{mathx}{\hyphenchar\font45}
\DeclareFontShape{U}{mathx}{m}{n}{
      <5> <6> <7> <8> <9> <10>
      <10.95> <12> <14.4> <17.28> <20.74> <24.88>
      mathx10
      }{}
\DeclareSymbolFont{mathx}{U}{mathx}{m}{n}
\DeclareMathSymbol{\bigtimes}{1}{mathx}{"91}

\def\s{\mathfrak{s}}

\providecommand{\figures}{false}
{ \ifthenelse{\equal{\figures}{false}} {#1}{\[ {\rm Figure \ missing !} \]} }{}

\DeclareMathOperator{\CDi}{\mathsf{CD}}
\DeclareMathOperator{\PL}{\mathsf{PreLie}}
\DeclareMathOperator{\Nov}{\mathsf{Nov}}
\DeclareMathOperator{\Mag}{\mathsf{Mag}}

\newcommand{\bn}{{\mathbf n}}

\newcommand{\cA}{{\mathcal A}}
\newcommand{\cP}{{\mathcal P}}

\def\CP{\mathcal{P}}

\def\CT{\mathcal{T}}

\tikzstyle{tinydots}=[dash pattern=on \pgflinewidth off \pgflinewidth]
\tikzstyle{superdense}=[dash pattern=on 4pt off 1pt]

\newcommand{\RT}{\mathcal{T}}

\newcommand{\mcR}{\mathcal{R}}

\newcommand{\mcL}{\mathcal{L}}

\newcommand{\length}[1]{\mathcal{l}({#1})}

\newcommand{\coord}{{\mcR}}

\newcommand{\beq}{\begin{equation}}
\newcommand{\eeq}{\end{equation}}

\usepackage{empheq}

\newcommand{\mfL}{\mathfrak{L}}

\newcommand{\mfl}{\mathfrak{l}}

\def\Lab{\mathfrak{L}}

\def\${|\!|\!|}

\newcounter{theorem}

\newtheorem{ex}[theorem]{Example}

\newenvironment{DIFnomarkup}{}{} % see man latexdiff

\newcommand{\I}{{\mathcal I}}

\theorembodyfont{\rmfamily}

\newcommand{\rrightarrow}{{\to\hskip -4.9mm\raise 1pt\hbox{$\to$}}}

\newfont{\indic}{bbmss12}

\def\Nabla_#1{\nabla_{\!#1}}

\makeatletter
\pgfdeclareshape{crosscircle}
{
  \inheritsavedanchors[from=circle] % this is nearly a circle
  \inheritanchorborder[from=circle]
  \inheritanchor[from=circle]{north}
  \inheritanchor[from=circle]{north west}
  \inheritanchor[from=circle]{north east}
  \inheritanchor[from=circle]{center}
  \inheritanchor[from=circle]{west}
  \inheritanchor[from=circle]{east}
  \inheritanchor[from=circle]{mid}
  \inheritanchor[from=circle]{mid west}
  \inheritanchor[from=circle]{mid east}
  \inheritanchor[from=circle]{base}
  \inheritanchor[from=circle]{base west}
  \inheritanchor[from=circle]{base east}
  \inheritanchor[from=circle]{south}
  \inheritanchor[from=circle]{south west}
  \inheritanchor[from=circle]{south east}
  \inheritbackgroundpath[from=circle]
  \foregroundpath{
    \centerpoint%
    \pgf@xc=\pgf@x%
    \pgf@yc=\pgf@y%
    \pgfutil@tempdima=\radius%
    \pgfmathsetlength{\pgf@xb}{\pgfkeysvalueof{/pgf/outer xsep}}%  
    \pgfmathsetlength{\pgf@yb}{\pgfkeysvalueof{/pgf/outer ysep}}%  
    \ifdim\pgf@xb<\pgf@yb%
      \advance\pgfutil@tempdima by-\pgf@yb%
    \else%
      \advance\pgfutil@tempdima by-\pgf@xb%
    \fi%
    \pgfpathmoveto{\pgfpointadd{\pgfqpoint{\pgf@xc}{\pgf@yc}}{\pgfqpoint{-0.707107\pgfutil@tempdima}{-0.707107\pgfutil@tempdima}}}
    \pgfpathlineto{\pgfpointadd{\pgfqpoint{\pgf@xc}{\pgf@yc}}{\pgfqpoint{0.707107\pgfutil@tempdima}{0.707107\pgfutil@tempdima}}}
    \pgfpathmoveto{\pgfpointadd{\pgfqpoint{\pgf@xc}{\pgf@yc}}{\pgfqpoint{-0.707107\pgfutil@tempdima}{0.707107\pgfutil@tempdima}}}
    \pgfpathlineto{\pgfpointadd{\pgfqpoint{\pgf@xc}{\pgf@yc}}{\pgfqpoint{0.707107\pgfutil@tempdima}{-0.707107\pgfutil@tempdima}}}
  }
}
\makeatother

\def\symbol#1{\textcolor{symbols}{#1}}

\def\decorate#1#2{
        \ifnum#2>0
    		\foreach \count in {1,...,#2}{
	       	let
				\p1 = (sourcenode.center),
                \p2 = (sourcenode.east),
				\n1 = {\x2-\x1},
				\n2 = {1mm},
				\n3 = {(1.3+0.6*(\count-1))*\n1},
				\n4 = {0.7*\n1}
			in 
        		node[rectangle,fill=symbols,rotate=30,inner sep=0pt,minimum width=0.2*\n2,minimum height=\n2] at ($(sourcenode.center) + (\n3,\n4)$) {}
				}
		\fi
        \ifnum#1>0
    		\foreach \count in {1,...,#1}{
	       	let
				\p1 = (sourcenode.center),
                \p2 = (sourcenode.east),
				\n1 = {\x2-\x1},
				\n2 = {1mm},
				\n3 = {(1.3+0.6*(\count-1))*\n1},
				\n4 = {0.7*\n1}
			in 
        		node[rectangle,fill=symbols,rotate=-30,inner sep=0pt,minimum width=0.2*\n2,minimum height=\n2] at ($(sourcenode.center) + (-\n3,\n4)$) {}
				}
		\fi
}

\tikzset{
    dectriangle/.style 2 args={
        triangle,
        alias=sourcenode,
        append after command={\decorate{#1}{#2}}
    },
    dectriangle/.default={0}{0},
}

\tikzset{
	cross/.style={path picture={ 
  		\draw[symbols]
			(path picture bounding box.south east) -- (path picture bounding box.north west) (path picture bounding box.south west) -- (path picture bounding box.north east);
		}},
root/.style={circle,fill=green!50!black,inner sep=0pt, minimum size=1.2mm},
        dot/.style={circle,fill=pageforeground,inner sep=0pt, minimum size=1mm},
        dotred/.style={circle,fill=pageforeground!50!pagebackground,inner sep=0pt, minimum size=2mm},
        var/.style={circle,fill=pageforeground!10!pagebackground,draw=pageforeground,inner sep=0pt, minimum size=3mm},
        kernel/.style={semithick,shorten >=2pt,shorten <=2pt},
        kernels/.style={snake=zigzag,shorten >=2pt,shorten <=2pt,segment amplitude=1pt,segment length=4pt,line before snake=2pt,line after snake=5pt,},
        rho/.style={densely dashed,semithick,shorten >=2pt,shorten <=2pt},
           testfcn/.style={dotted,semithick,shorten >=2pt,shorten <=2pt},
        renorm/.style={shape=circle,fill=pagebackground,inner sep=1pt},
        labl/.style={shape=rectangle,fill=pagebackground,inner sep=1pt},
        xic/.style={very thin,circle,draw=symbols,fill=symbols,inner sep=0pt,minimum size=1.2mm},
        g/.style={very thin,rectangle,draw=symbols,fill=symbols!10!pagebackground,inner sep=0pt,minimum width=2.5mm,minimum height=1.2mm},
        xi/.style={very thin,circle,draw=symbols,fill=symbols!10!pagebackground,inner sep=0pt,minimum size=1.2mm},
	xies/.style={very thin,rectangle,fill=green!50!black!25,draw=symbols,inner sep=0pt,minimum size=1.1mm},
	xiesf/.style={very thin,rectangle,fill=green!50!black,draw=symbols,inner sep=0pt,minimum size=1.1mm},
        xix/.style={very thin,crosscircle,fill=symbols!10!pagebackground,draw=symbols,inner sep=0pt,minimum size=1.2mm},
        X/.style={very thin,cross,rectangle,fill=pagebackground,draw=symbols,inner sep=0pt,minimum size=1.2mm},
	xib/.style={thin,circle,fill=symbols!10!pagebackground,draw=symbols,inner sep=0pt,minimum size=1.6mm},
	xie/.style={thin,circle,fill=green!50!black,draw=symbols,inner sep=0pt,minimum size=1.6mm},
	xid/.style={thin,circle,fill=symbols,draw=symbols,inner sep=0pt,minimum size=1.6mm},
	xibx/.style={thin,crosscircle,fill=symbols!10!pagebackground,draw=symbols,inner sep=0pt,minimum size=1.6mm},
	kernels2/.style={very thick,draw=connection,segment length=12pt},
	keps/.style={thin,draw=symbols,->},
	kepspr/.style={thick,draw=connection,->},
	krho/.style={thin,draw=symbols,superdense,->},
	krhopr/.style={thick,draw=connection,superdense},
	triangle/.style = { regular polygon, regular polygon sides=3},
	not/.style={thin,circle,draw=connection,fill=connection,inner sep=0pt,minimum size=0.5mm},
	diff/.style = {very thin,draw=symbols,triangle,fill=red!50!black,inner sep=0pt,minimum size=1.6mm},
	diff1/.style = {very thin,dectriangle={1}{0},fill=red!50!black,draw=symbols,inner sep=0pt,minimum size=1.6mm},
	diff2/.style = {very thin,dectriangle={1}{1},fill=red!50!black,draw=symbols,inner sep=0pt,minimum size=1.6mm},
		diffmini/.style = {very thin,rectangle,fill=black,draw=black,inner sep=0pt,minimum size=0.75mm},
	 kernelsmod/.style={very thick,draw=connection,segment length=12pt},
	 rec/.style = {very thin,rectangle,fill=black,draw=black,inner sep=0pt,minimum size=2mm},
	cerc/.style={very thin,circle,draw=black,fill=symbols,inner sep=0pt,minimum size=2mm},
	stars/.style={very thin,star,star points=6,star point ratio=0.5, draw=black,fill=red,inner sep=0pt,minimum size=0.7mm},
	>=stealth,
        }
        \tikzset{
root/.style={circle,fill=black!50,inner sep=0pt, minimum size=3mm},
        circ/.style={circle,fill=white,draw=black,very thin,inner sep=.5pt, minimum size=1.2mm},
        round1/.style={fill=white,outer sep = 0,inner sep=2pt,rounded corners=1mm,draw,text=black,thin,minimum size=1.2mm},
          circ1/.style={circle,fill=red!10,draw=red,very thin,inner sep=.5pt, minimum size=1.2mm},
        rect/.style={fill=white,outer sep = 0,inner sep=2pt,rectangle,draw,text=black,thin,minimum size=1.2mm},
        rect1/.style={fill=white,outer sep = 0,inner sep=2pt,rectangle,draw,text=black,thin,minimum size=1.2mm},
        round2/.style={fill=red!10,outer sep = 0,inner sep=2pt,rounded corners=1mm,draw,text=black,thin,minimum size=1.2mm},
       round3/.style={fill=blue!10,outer sep = 0,inner sep=2pt,rounded corners=1mm,draw,text=black,thin,minimum size=1.2mm}, 
        rect2/.style={fill=black!10,outer sep = 0,inner sep=2pt,rectangle,draw,text=black,thin,minimum size=1.2mm},
        dot/.style={circle,fill=black,inner sep=0pt, minimum size=1.2mm},
        dotred/.style={circle,fill=black!50,inner sep=0pt, minimum size=2mm},
        var/.style={circle,fill=black!10,draw=black,inner sep=0pt, minimum size=3mm},
        kernel/.style={semithick,shorten >=2pt,shorten <=2pt},
         diag/.style={thin,shorten >=4pt,shorten <=4pt},
        kernel1/.style={thick},
        kernels/.style={snake=zigzag,shorten >=2pt,shorten <=2pt,segment amplitude=1pt,segment length=4pt,line before snake=2pt,line after snake=5pt,},
		kernels1/.style={snake=zigzag,segment amplitude=0.5pt,segment length=2pt},
		rho1/.style={densely dotted,semithick},
        rho/.style={densely dashed,semithick,shorten >=2pt,shorten <=2pt},
           testfcn/.style={dotted,semithick,shorten >=2pt,shorten <=2pt},
           visible/.style={draw, circle, fill, inner sep=0.25ex},
        renorm/.style={shape=circle,fill=white,inner sep=1pt},
        labl/.style={shape=rectangle,fill=white,inner sep=1pt},
        xic/.style={very thin,circle,fill=symbols,draw=black,inner sep=0pt,minimum size=1.2mm},
        xi/.style={very thin,circle,fill=blue!10,draw=black,inner sep=0pt,minimum size=1.2mm},
	xib/.style={very thin,circle,fill=blue!10,draw=black,inner sep=0pt,minimum size=1.6mm},
	xie/.style={very thin,circle,fill=green!50!black,draw=black,inner sep=0pt,minimum size=1mm},
	xid/.style={very thin,circle,fill=symbols,draw=black,inner sep=0pt,minimum size=1.6mm},
	edgetype/.style={very thin,circle,draw=black,inner sep=0pt,minimum size=5mm},
	nodetype/.style={very thick,circle,draw=black,inner sep=0pt,minimum size=5mm},
	kernels2/.style={very thick,draw=connection,segment length=12pt},
clean/.style={thin,circle,fill=black,inner sep=0pt,minimum size=1mm},	not/.style={thin,circle,fill=symbols,draw=connection,fill=connection,inner sep=0pt,minimum size=0.8mm},
	>=stealth,
        }

\makeatletter
\def\DeclareSymbol#1#2#3{%
	\expandafter\gdef\csname MH@symb@#1\endcsname{\tikzsetnextfilename{symbol#1}%
	\tikz[baseline=#2,scale=0.15,draw=symbols,line join=round]{#3}}%
	\expandafter\gdef\csname MH@symb@#1s\endcsname{\scalebox{0.75}{\tikzsetnextfilename{symbol#1}%
	\tikz[baseline=#2,scale=0.15,draw=symbols,line join=round]{#3}}}%
	\expandafter\gdef\csname MH@symb@#1ss\endcsname{\scalebox{0.65}{\tikzsetnextfilename{symbol#1}%
	\tikz[baseline=#2,scale=0.15,draw=symbols,line join=round]{#3}}}%
	}
\def\<#1>{\ifthenelse{\boolean{mmode}}{\mathchoice{\csname MH@symb@#1\endcsname}{\csname MH@symb@#1\endcsname}{\csname MH@symb@#1s\endcsname}{\csname MH@symb@#1ss\endcsname}}{\csname MH@symb@#1\endcsname}}
\makeatother

\DeclareSymbol{Xi22}{0.5}{\draw (0,0) node[xi] {} -- (-1,1) node[not] {} -- (0,2) node[xi] {};} % 1 not used in the text
\DeclareSymbol{Xi2}{-2}{\draw (-1,-0.25) node[xi] {} -- (0,1) node[xi] {};} % 2
\DeclareSymbol{Xi2b}{-2}{\draw (-1,-0.25) node[xic] {} -- (0,1) node[xic] {};} % 2
\DeclareSymbol{Xi2g}{-2}{\draw (-1,-0.25) node[xies] {} -- (0,1) node[xi] {};} % 2
\DeclareSymbol{Xi2g2}{-2}{\draw (-1,-0.25) node[xi] {} -- (0,1) node[xies] {};} % 2
\DeclareSymbol{cXi2}{-2}{\draw (0,-0.25) node[xi] {} -- (-1,1) node[xic] {};}%3
\DeclareSymbol{Xi3}{0}{\draw (0,0) node[xi] {} -- (-1,1) node[xi] {} -- (0,2) node[xi] {};}%4
\DeclareSymbol{XiIIXi}{0}{\draw (0,0) node[xi] {} -- (-1,1); \draw[kernels2] (-1,1) node[not] {} -- (0,2) node[xi] {};}%4

\DeclareSymbol{Xi4}{2}{\draw (0,0) node[xi] {} -- (-1,1) node[xi] {} -- (0,2) node[xi] {} -- (-1,3) node[xi] {};}%5
\DeclareSymbol{Xi4_1}{2}{\draw (0,0) node[xic] {} -- (-1,1) node[xic] {} -- (0,2) node[xi] {} -- (-1,3) node[xi] {};}%6
\DeclareSymbol{Xi4_2}{2}{\draw (0,0) node[xic] {} -- (-1,1) node[xi] {} -- (0,2) node[xi] {} -- (-1,3) node[xic] {};}%7
\DeclareSymbol{Xi2X}{-2}{\draw (0,-0.25) node[xi] {} -- (-1,1) node[xix] {};}%8
\DeclareSymbol{XXi2}{-2}{\draw (0,-0.25) node[xix] {} -- (-1,1) node[xi] {};}%9
\DeclareSymbol{IIXi}{0}{\draw (0,-0.25) node[not] {} -- (-1,1) node[xi] {} -- (0,2) node[xi] {};}%10
\DeclareSymbol{IXi^2}{-1}{\draw (-1,1) node[xi] {} -- (0,0) node[not] {} -- (1,1) node[xi] {};}%11
\DeclareSymbol{IIXi^2}{-4}{\draw (0,-1.5) node[not] {} -- (0,0);
\draw[kernels2] (-1,1) node[xi] {} -- (0,0) node[not] {} -- (1,1) node[xi] {};}%12
\DeclareSymbol{XiX}{-2.8}{\node[xibx] {};}%13
\DeclareSymbol{tauX}{-2.8}{ \node[X] {};}%14
\DeclareSymbol{Xi}{-2.8}{\node[xib] {};}%15

\DeclareSymbol{IXiX}{-1}{\draw (0,-0.25) node[not] {} -- (-1,1) node[xix] {};}%18
\DeclareSymbol{IXi3}{2}{\draw (0,-0.25) node[not] {} -- (-1,1) node[xi] {} -- (0,2) node[xi] {} -- (-1,3) node[xi] {};}%20
\DeclareSymbol{IXi}{-2}{\draw (0,-0.25) node[not] {} -- (-1,1) node[xi] {};}%21
\DeclareSymbol{XiI}{-2}{\draw (0,-0.25) node[xi] {} -- (-1,1) node[not] {};}%22

\DeclareSymbol{Xi4b}{0}{\draw(0,1.5) node[xi] {} -- (0,0); \draw (-1,1) node[xi] {} -- (0,0) node[xi] {} -- (1,1) node[xi] {};}%23
\DeclareSymbol{Xi4b'}{0}{\draw(0,1.5) node[xi] {} -- (0,-0.2); \draw (-1,1) node[xi] {} -- (0,-0.2) node[not] {} -- (1,1) node[xi] {};}%24 not used
\DeclareSymbol{Xi4c}{0}{\draw (0,1) -- (0.8,2.2) node[xi] {};\draw (0,-0.25) node[xi] {} -- (0,1) node[xi] {} -- (-0.8,2.2) node[xi] {};}%25
\DeclareSymbol{Xi4d}{-4.5}{\draw (0,-1.5) node[not] {} -- (0,0); \draw (-1,1) node[xi] {} -- (0,0) node[xi] {} -- (1,1) node[xi] {};}%26 not used
\DeclareSymbol{Xi4e}{0}{\draw (0,2) node[xi] {} -- (-1,1) node[xi] {} -- (0,0) node[xi] {} -- (1,1) node[xi] {};}%27
\DeclareSymbol{Xi4e'}{0}{\draw (0,2) node[xi] {} -- (-1,1) node[xi] {} -- (0,-0.2) node[not] {} -- (1,1) node[xi] {};}%28 not used

\DeclareSymbol{Xitwo}%29
{0}{\draw[kernels2] (0,0) node[not] {} -- (-1,1) node[not] {}
-- (-2,2) node[not]{} -- (-3,3) node[xi]  {};
\draw[kernels2] (0,0) -- (1,1) node[xi] {};
\draw[kernels2] (-1,1) -- (0,2) node[xi] {};
\draw[kernels2] (-2,2) -- (-1,3) node[xi] {};}

\DeclareSymbol{IXitwo}%30
{0}{\draw (-.7,1.2) node[xi] {} -- (0,-0.2) -- (.7,1.2) node[xi] {};}
\DeclareSymbol{I1Xitwo}%30
{0}{\draw[kernels2] (0,0) node[not] {} -- (-1,1) node[xi] {};
\draw[kernels2] (0,0) -- (1,1) node[xi] {};}

\DeclareSymbol{I1Xitwobis}%30
{0}{\draw[kernels2] (0,0) node[not] {} -- (-1,1) node[xies] {};
\draw[kernels2] (0,0) -- (1,1) node[xies] {};}

\DeclareSymbol{I1Xitwog}%30
{0}{\draw[kernels2] (0,0) node[not] {} -- (-1,1) node[xies] {};
\draw[kernels2] (0,0) -- (1,1) node[xi] {};}

\DeclareSymbol{cI1Xitwo}%31
{0}{\draw[kernels2] (0,0) node[not] {} -- (-1,1) node[xic] {};
\draw[kernels2] (0,0) -- (1,1) node[xi] {};}

\DeclareSymbol{I1IXi3}{0}{\draw (0,0) node[xi] {} -- (-1,1) ; %32
\draw[kernels2] (-1,1) node[not] {} -- (0,2) node[xi] {};
\draw[kernels2] (-1,1) node[not] {} -- (-2,2) node[xi] {};}

\DeclareSymbol{I1Xi3c}{-1}{\draw[kernels2](0,1.5) node[xi] {} -- (0,0) node[not] {}; \draw (-1,1) node[xi] {} -- (0,0) ; \draw[kernels2] (0,0) -- (1,1) node[xi] {};}%33

\DeclareSymbol{I1Xi3cbis}{-1}{\draw[kernels2](0,1.5) node[xies] {} -- (0,0) node[not] {}; \draw (-1,1) node[xies] {} -- (0,0) ; \draw[kernels2] (0,0) -- (1,1) node[xies] {};}%33

\DeclareSymbol{I1IXi3b}{0}{\draw[kernels2] (0,0) node[not] {} -- (-1,1) ; \draw[kernels2] (0,0)   -- (1,1) node[xi] {} ;
\draw (-1,1) node[xi] {} -- (0,2) node[xi] {};
}%34

\DeclareSymbol{I1IXi3c}{0}{\draw[kernels2] (0,0) node[not] {} -- (-1,1) ; \draw[kernels2] (0,0)   -- (1,1) node[xi] {} ;
\draw[kernels2] (-1,1) node[not] {} -- (0,2) node[xi] {};
\draw[kernels2] (-1,1) node[not] {} -- (-2,2) node[xi] {};}%35

\DeclareSymbol{I1IXi3cbis}{0}{\draw[kernels2] (0,0) node[not] {} -- (-1,1) ; \draw[kernels2] (0,0)   -- (1,1) node[xies] {} ;
\draw[kernels2] (-1,1) node[not] {} -- (0,2) node[xies] {};
\draw[kernels2] (-1,1) node[not] {} -- (-2,2) node[xies] {};}%35

\DeclareSymbol{I1Xi}{0}{\draw[kernels2] (0,0) node[not] {} -- (-1,1)  node[xi] {} ;}%36

\DeclareSymbol{I1Xi4a}{2}{\draw[kernels2] (0,0) node[not] {} -- (-1,1) ; \draw[kernels2] (0,0) node[not] {} -- (1,1) node[xi] {} ;%37
\draw (-1,1) node[xi] {} -- (0,2) node[xi] {} -- (-1,3) node[xi] {};}
\DeclareSymbol{cI1Xi4a}{2}{\draw[kernels2] (0,0) node[not] {} -- (-1,1) ; \draw[kernels2] (0,0) node[not] {} -- (1,1) node[xic] {} ;%39
\draw (-1,1) node[xic] {} -- (0,2) node[xi] {} -- (-1,3) node[xi] {};}
\DeclareSymbol{I1Xi4b}{2}{\draw (0,0) node[xi] {} -- (-1,1) node[xi] {} -- (0,2) ; \draw[kernels2] (0,2) node[not] {} -- (-1,3) node[xi] {};\draw[kernels2] (0,2)  -- (1,3) node[xi] {};
}%41

\DeclareSymbol{cI1Xi4b}{2}{\draw (0,0) node[xic] {} -- (-1,1) node[xic] {} -- (0,2) ; \draw[kernels2] (0,2) node[not] {} -- (-1,3) node[xi] {};\draw[kernels2] (0,2)  -- (1,3) node[xi] {};
}%42

\DeclareSymbol{I1Xi4c}{2}{\draw (0,0) node[xi] {} -- (-1,1) node[not] {}; \draw[kernels2] (-1,1) -- (0,2) ; 
\draw[kernels2] (-1,1) -- (-2,2) node[xi] {} ;
\draw (0,2) node[xi] {} -- (-1,3) node[xi] {};}%43

\DeclareSymbol{cI1Xi4c}{2}{\draw (0,0) node[xic] {} -- (-1,1) node[not] {}; \draw[kernels2] (-1,1) -- (0,2) ; 
\draw[kernels2] (-1,1) -- (-2,2) node[xic] {} ;
\draw (0,2) node[xi] {} -- (-1,3) node[xi] {};}%44

\DeclareSymbol{I1Xi4ab}{2}{\draw[kernels2] (0,0) node[not] {} -- (-1,1) ; \draw[kernels2] (0,0) node[not] {} -- (1,1) node[xi] {};\draw (-1,1) node[xi] {} -- (0,2) ; \draw[kernels2] (0,2) node[not] {} -- (-1,3) node[xi] {};\draw[kernels2] (0,2)  -- (1,3) node[xi] {}; }%45

\DeclareSymbol{cI1Xi4ab}{2}{\draw[kernels2] (0,0) node[not] {} -- (-1,1) ; \draw[kernels2] (0,0) node[not] {} -- (1,1) node[xic] {};\draw (-1,1) node[xic] {} -- (0,2) ; \draw[kernels2] (0,2) node[not] {} -- (-1,3) node[xi] {};\draw[kernels2] (0,2)  -- (1,3) node[xi] {}; }%46

\DeclareSymbol{I1Xi4bc}{2}{\draw (0,0) node[xi] {} -- (-1,1) node[not] {}; \draw[kernels2] (-1,1) -- (0,2) ; %47
\draw[kernels2] (-1,1) -- (-2,2) node[xi] {} ; \draw[kernels2] (0,2) node[not] {} -- (-1,3) node[xi] {};\draw[kernels2] (0,2)  -- (1,3) node[xi] {};
}

\DeclareSymbol{cI1Xi4bc}{2}{\draw (0,0) node[xic] {} -- (-1,1) node[not] {}; \draw[kernels2] (-1,1) -- (0,2) ; %48
\draw[kernels2] (-1,1) -- (-2,2) node[xic] {} ; \draw[kernels2] (0,2) node[not] {} -- (-1,3) node[xi] {};\draw[kernels2] (0,2)  -- (1,3) node[xi] {};
}

\DeclareSymbol{I1Xi4abcc1}{2}{\draw[kernels2] (0,0) node[not] {} -- (-1,1) node[not] {}%50
-- (-2,2) node[not]{} -- (-3,3) node[xic]  {};
\draw[kernels2] (0,0) -- (1,1) node[xic] {};
\draw[kernels2] (-1,1) -- (0,2) node[xi] {};
\draw[kernels2] (-2,2) -- (-1,3) node[xi] {};
}

\DeclareSymbol{I1Xi4abcc1b}{2}{\draw[kernels2] (0,0) node[not] {} -- (-1,1) node[not] {}%50
-- (-2,2) node[not]{} -- (-3,3) node[xi]  {};
\draw[kernels2] (0,0) -- (1,1) node[xic] {};
\draw[kernels2] (-1,1) -- (0,2) node[xic] {};
\draw[kernels2] (-2,2) -- (-1,3) node[xi] {};
}

\DeclareSymbol{I1Xi4abcc2}{2}{\draw[kernels2] (0,0) node[not] {} -- (-1,1) node[not] {}%51
-- (-2,2) node[not]{} -- (-3,3) node[xic]  {};
\draw[kernels2] (0,0) -- (1,1) node[xi] {};
\draw[kernels2] (-1,1) -- (0,2) node[xi] {};
\draw[kernels2] (-2,2) -- (-1,3) node[xic] {};
}

\DeclareSymbol{I1Xi4ac}{2}{\draw[kernels2] (0,0) node[not] {} -- (-1,1) ; \draw[kernels2] (0,0) node[not] {} -- (1,1) node[xi] {}; 
\draw[kernels2] (-1,1) node[not] {} -- (0,2) ; %52
\draw[kernels2] (-1,1) -- (-2,2) node[xi] {} ;
\draw (0,2) node[xi] {} -- (-1,3) node[xi] {};}

\DeclareSymbol{cI1Xi4ac}{2}{\draw[kernels2] (0,0) node[not] {} -- (-1,1) ; \draw[kernels2] (0,0) node[not] {} -- (1,1) node[xic] {}; 
\draw[kernels2] (-1,1) node[not] {} -- (0,2) ; %53
\draw[kernels2] (-1,1) -- (-2,2) node[xic] {} ;
\draw (0,2) node[xi] {} -- (-1,3) node[xi] {};}

\DeclareSymbol{I1Xi4acc1}{2}{\draw[kernels2] (0,0) node[not] {} -- (-1,1) ; \draw[kernels2] (0,0) node[not] {} -- (1,1) node[xic] {}; 
\draw[kernels2] (-1,1) node[not] {} -- (0,2) ; %54
\draw[kernels2] (-1,1) -- (-2,2) node[xi] {} ;
\draw (0,2) node[xic] {} -- (-1,3) node[xi] {};}

\DeclareSymbol{I1Xi4acc2}{2}{\draw[kernels2] (0,0) node[not] {} -- (-1,1) ; \draw[kernels2] (0,0) node[not] {} -- (1,1) node[xic] {}; 
\draw[kernels2] (-1,1) node[not] {} -- (0,2) ; %55
\draw[kernels2] (-1,1) -- (-2,2) node[xi] {} ;
\draw (0,2) node[xi] {} -- (-1,3) node[xic] {};}

\DeclareSymbol{2I1Xi4}{2}{\draw[kernels2] (0,0) node[not] {} -- (-1,1) node[not] {};
\draw[kernels2] (0,0) -- (1,1) node[not] {};%56
\draw[kernels2] (-1,1) -- (-1.5,2.5) node[xi] {};
\draw[kernels2] (-1,1) -- (-0.5,2.5) node[xi] {};
\draw[kernels2] (1,1) -- (0.5,2.5) node[xi] {};
\draw[kernels2] (1,1) -- (1.5,2.5) node[xi] {};
}

\DeclareSymbol{2I1Xi4dis}{2}{\draw[kernels2] (0,0) node[not] {} -- (-1,1) node[not] {};
\draw[kernels2] (0,0) -- (1,1) node[not] {};%56
\draw[kernels2] (-1,1) -- (-1.5,2.5) node[xies] {};
\draw[kernels2] (-1,1) -- (-0.5,2.5) node[xies] {};
\draw[kernels2] (1,1) -- (0.5,2.5) node[xies] {};
\draw[kernels2] (1,1) -- (1.5,2.5) node[xies] {};
}

\DeclareSymbol{2I1Xi4c1}{2}{\draw[kernels2] (0,0) node[not] {} -- (-1,1) node[not] {};%57
\draw[kernels2] (0,0) -- (1,1) node[not] {};
\draw[kernels2] (-1,1) -- (-1.5,2.5) node[xic] {};
\draw[kernels2] (-1,1) -- (-0.5,2.5) node[xi] {};
\draw[kernels2] (1,1) -- (0.5,2.5) node[xic] {};
\draw[kernels2] (1,1) -- (1.5,2.5) node[xi] {};
}

\DeclareSymbol{2I1Xi4c2}{2}{\draw[kernels2] (0,0) node[not] {} -- (-1,1) node[not] {};%58
\draw[kernels2] (0,0) -- (1,1) node[not] {};
\draw[kernels2] (-1,1) -- (-1.5,2.5) node[xic] {};
\draw[kernels2] (-1,1) -- (-0.5,2.5) node[xic] {};
\draw[kernels2] (1,1) -- (0.5,2.5) node[xi] {};
\draw[kernels2] (1,1) -- (1.5,2.5) node[xi] {};
}

\DeclareSymbol{2I1Xi4b}{2}{\draw[kernels2] (0,0) node[not] {} -- (-1,1) ;%59
\draw[kernels2] (0,0) -- (1,1);
\draw (-1,1) node[xi] {} -- (-1,2.5) node[xi] {};
\draw (1,1)  node[xi] {} -- (1,2.5) node[xi] {};
}

\DeclareSymbol{2I1Xi4bb}{2}{\draw[kernels2] (0,0) node[not] {} -- (-1,1) ;%59
\draw[kernels2] (0,0) -- (1,1);
\draw (-1,1) node[xi] {} -- (-1,2.5) node[xiesf] {};
\draw (1,1)  node[xi] {} -- (1,2.5) node[xic] {};
}

\DeclareSymbol{2I1Xi4c}{2}{\draw[kernels2] (0,0) node[not] {} -- (-1,1);%60
\draw[kernels2] (0,0) -- (1,1) node[not] {};
\draw (-1,1)  node[xi] {} -- (-1,2.5) node[xi] {};
\draw[kernels2] (1,1) -- (0.4,2.5) node[xi] {};
\draw[kernels2] (1,1) -- (1.6,2.5) node[xi] {};
}

\DeclareSymbol{2I1Xi4cc1}{2}{\draw[kernels2] (0,0) node[not] {} -- (-1,1);%61
\draw[kernels2] (0,0) -- (1,1) node[not] {};
\draw (-1,1)  node[xic] {} -- (-1,2.5) node[xi] {};
\draw[kernels2] (1,1) -- (0.4,2.5) node[xic] {};
\draw[kernels2] (1,1) -- (1.6,2.5) node[xi] {};
}

\DeclareSymbol{2I1Xi4cc2}{2}{\draw[kernels2] (0,0) node[not] {} -- (-1,1);%62
\draw[kernels2] (0,0) -- (1,1) node[not] {};
\draw (-1,1)  node[xic] {} -- (-1,2.5) node[xic] {};
\draw[kernels2] (1,1) -- (0.4,2.5) node[xi] {};
\draw[kernels2] (1,1) -- (1.6,2.5) node[xi] {};
}

\DeclareSymbol{Xi4ba}{0}{\draw(-0.5,1.5) node[xi] {} -- (0,0); \draw (-1.5,1) node[xi] {} -- (0,0) node[not] {}; \draw[kernels2] (0,0) -- (1.5,1) node[xi] {};
\draw[kernels2] (0,0) -- (0.5,1.5) node[xi] {} ;}%63

\DeclareSymbol{Xi4badis}{0}{\draw(-0.5,1.5) node[xies] {} -- (0,0); \draw (-1.5,1) node[xies] {} -- (0,0) node[not] {}; \draw[kernels2] (0,0) -- (1.5,1) node[xies] {};
\draw[kernels2] (0,0) -- (0.5,1.5) node[xies] {} ;}%63

\DeclareSymbol{Xi4ba1}{0}{\draw(-0.5,1.5) node[xi] {} -- (0,0); \draw (-1.5,1) node[xi] {} -- (0,0) node[not] {}; \draw[kernels2] (0,0) -- (1.5,1) node[xic] {};
\draw[kernels2] (0,0) -- (0.5,1.5) node[xic] {} ;}%64

\DeclareSymbol{Xi4ba1b}{0}{\draw(-0.5,1.5) node[xic] {} -- (0,0); \draw (-1.5,1) node[xic] {} -- (0,0) node[not] {}; \draw[kernels2] (0,0) -- (1.5,1) node[xi] {};
\draw[kernels2] (0,0) -- (0.5,1.5) node[xi] {} ;}%64

\DeclareSymbol{Xi4ba1bdiff}{0}{\draw(-0.5,1.5) node[xic] {} -- (0,0); \draw (-1.5,1) node[xic] {} -- (0,0) node[not] {}; \draw (0,0) -- (1.5,1) node[xi] {};
\draw (0,0) -- (0.5,1.5) node[xi] {};
\draw(0,0) node[diff] {};}%64

\DeclareSymbol{Xi4ba1bb}{0}{\draw(-0.5,1.5) node[xic] {} -- (0,0); \draw (-1.5,1) node[xiesf] {} -- (0,0) node[not] {}; \draw[kernels2] (0,0) -- (1.5,1) node[xi] {};
\draw[kernels2] (0,0) -- (0.5,1.5) node[xi] {} ;}%64

\DeclareSymbol{Xi4ba2}{0}{\draw(-0.5,1.5) node[xi] {} -- (0,0); \draw (-1.5,1) node[xic] {} -- (0,0) node[not] {}; \draw[kernels2] (0,0) -- (1.5,1) node[xi] {};
\draw[kernels2] (0,0) -- (0.5,1.5) node[xic] {} ;}%65

\DeclareSymbol{Xi4ba2b}{0}{\draw(-0.5,1.5) node[xi] {} -- (0,0); \draw (-1.5,1) node[xic] {} -- (0,0) node[not] {}; \draw[kernels2] (0,0) -- (1.5,1) node[xi] {};
\draw[kernels2] (0,0) -- (0.5,1.5) node[xiesf] {} ;}%65

\DeclareSymbol{Xi4ca}{0}{\draw (0,1) -- (-1,2.2) node[xi] {};\draw (0,-0.25) node[xi] {} -- (0,1) ; \draw[kernels2] (0,1) node[not] {} -- (1,2.2) node[xi] {};%67
\draw[kernels2] (0,1) {} -- (0,2.7) node[xi] {};
}

\DeclareSymbol{Xi4cb}{0}{\draw (-1,1) -- (-2,2) node[xi] {};\draw[kernels2] (0,0)  -- (-1,1) node[xi] {} ; \draw[kernels2] (0,0) node[not] {} -- (1,1) node[xi] {} ; 
\draw (-1,1) node[xi] {} -- (0,2) node[xi] {};}

\DeclareSymbol{Xi4cbb}{0}{\draw (-1,1) -- (-2,2) node[xiesf] {};\draw[kernels2] (0,0)  -- (-1,1) node[xi] {} ; \draw[kernels2] (0,0) node[not] {} -- (1,1) node[xi] {} ; 
\draw (-1,1) node[xi] {} -- (0,2) node[xic] {};}

\DeclareSymbol{Xi4cbc1}{0}{\draw (-1,1) -- (-2,2) node[xic] {};\draw[kernels2] (0,0)  -- (-1,1) node[xic] {} ; \draw[kernels2] (0,0) node[not] {} -- (1,1) node[xi] {} ; 
\draw (-1,1) node[xic] {} -- (0,2) node[xi] {};}

\DeclareSymbol{Xi4cbc2}{0}{\draw (-1,1) -- (-2,2) node[xi] {};\draw[kernels2] (0,0)  -- (-1,1) node[xi] {} ; \draw[kernels2] (0,0) node[not] {} -- (1,1) node[xic] {} ; 
\draw (-1,1) node[xic] {} -- (0,2) node[xi] {};}

\DeclareSymbol{Xi4cab}{0}{\draw (-1,1) -- (-2,2) node[xi] {};\draw[kernels2] (0,0)  -- (-1,1); \draw[kernels2] (0,0) node[not] {} -- (1,1) node[xi] {} ; 
\draw[kernels2] (-1,1)  {} -- (0,2) node[xi] {};
\draw[kernels2] (-1,1) node[not] {} -- (-1,2.5) node[xi] {};
}%69

\DeclareSymbol{Xi4cabdis}{0}{\draw (-1,1) -- (-2,2) node[xies] {};\draw[kernels2] (0,0)  -- (-1,1); \draw[kernels2] (0,0) node[not] {} -- (1,1) node[xies] {} ; 
\draw[kernels2] (-1,1)  {} -- (0,2) node[xies] {};
\draw[kernels2] (-1,1) node[not] {} -- (-1,2.5) node[xies] {};
}%69

\DeclareSymbol{Xi4cabc1}{0}{\draw (-1,1) -- (-2,2) node[xi] {};\draw[kernels2] (0,0)  -- (-1,1); \draw[kernels2] (0,0) node[not] {} -- (1,1) node[xic] {} ; 
\draw[kernels2] (-1,1)  {} -- (0,2) node[xic] {};
\draw[kernels2] (-1,1) node[not] {} -- (-1,2.5) node[xi] {};
}%69

\DeclareSymbol{Xi4cabc2}{0}{\draw (-1,1) -- (-2,2) node[xic] {};\draw[kernels2] (0,0)  -- (-1,1); \draw[kernels2] (0,0) node[not] {} -- (1,1) node[xic] {} ; 
\draw[kernels2] (-1,1)  {} -- (0,2) node[xi] {};
\draw[kernels2] (-1,1) node[not] {} -- (-1,2.5) node[xi] {};
}%69

\DeclareSymbol{Xi4ea}{1.5}{\draw (-1,2.5) node[xi] {} -- (-1,1) node[xi] {} -- (0,0); %71
 \draw[kernels2] (0,0)  -- (1,1) node[xi] {};
\draw[kernels2] (0,0) node[not] {} -- (0,1.5) node[xi] {}; }

\DeclareSymbol{Xi4eac1}{1.5}{\draw (-1,2.5) node[xic] {} -- (-1,1) node[xi] {} -- (0,0); %72
 \draw[kernels2] (0,0)  -- (1,1) node[xic] {};
\draw[kernels2] (0,0) node[not] {} -- (0,1.5) node[xi] {}; }

\DeclareSymbol{Xi4eac1b}{1.5}{\draw (-1,2.5) node[xic] {} -- (-1,1) node[xi] {} -- (0,0); %72
 \draw[kernels2] (0,0)  -- (1,1) node[xiesf] {};
\draw[kernels2] (0,0) node[not] {} -- (0,1.5) node[xi] {}; }

\DeclareSymbol{Xi4eac2}{1.5}{\draw (-1,2.5) node[xic] {} -- (-1,1) node[xic] {} -- (0,0); %73
 \draw[kernels2] (0,0)  -- (1,1) node[xi] {};
\draw[kernels2] (0,0) node[not] {} -- (0,1.5) node[xi] {}; }

\DeclareSymbol{Xi4eact1}{1.5}{\draw (-1,2.5) node[xic] {} -- (-1,1) node[xi] {} -- (0,0); %74
 \draw (0,0)  -- (1,1) node[xic] {};
\draw[rho] (0,0) node[not] {} -- (0,1.5) node[xi] {}; }

\DeclareSymbol{Xi4eact2}{1.5}{\draw[rho] (-1,2.5) node[xic] {} -- (-1,1) node[xi] {} -- (0,0); %75
 \draw (0,0)  -- (1,1) node[xic] {};
\draw (0,0) node[not] {} -- (0,1.5) node[xi] {}; }

\DeclareSymbol{Xi4eabis}{1.5}{\draw (-1,2.5) node[xi] {} -- (-1,1) ; \draw[kernels2] (-1,1) node[xi] {} -- (0,0); 
 \draw (0,0)  -- (1,1) node[xi] {};%76
\draw[kernels2] (0,0) node[not] {} -- (0,1.5) node[xi] {}; }

\DeclareSymbol{Xi4eabisc1}{1.5}{\draw (-1,2.5) node[xic] {} -- (-1,1) ; \draw[kernels2] (-1,1) node[xi] {} -- (0,0); 
 \draw (0,0)  -- (1,1) node[xi] {};%77
\draw[kernels2] (0,0) node[not] {} -- (0,1.5) node[xic] {}; }

\DeclareSymbol{Xi4eabisc1b}{1.5}{\draw (-1,2.5) node[xic] {} -- (-1,1) ; \draw[kernels2] (-1,1) node[xi] {} -- (0,0); 
 \draw (0,0)  -- (1,1) node[xi] {};%77
\draw[kernels2] (0,0) node[not] {} -- (0,1.5) node[xiesf] {}; }

\DeclareSymbol{Xi4eabisc1bis}{1.5}{\draw (-1,2.5) node[xi] {} -- (-1,1) ; \draw[kernels2] (-1,1) node[xi] {} -- (0,0); 
 \draw (0,0)  -- (1,1) node[xi] {};%78
\draw[kernels2] (0,0) node[not] {} -- (0,1.5) node[xi] {};
\draw (-2,1) node[] {\tiny{$i$}};
\draw (-2,2.5) node[] {\tiny{$\ell$}};
\draw (2,1) node[] {\tiny{$k$}};
\draw (0,2.5) node[] {\tiny{$j$}};
 }

\DeclareSymbol{Xi4eabisc1tris}{1.5}{\draw (-1,2.5) node[xi] {} -- (-1,1) ; \draw[kernels2] (-1,1) node[xi] {} -- (0,0); 
 \draw (0,0)  -- (1,1) node[xi] {};%78
\draw[kernels2] (0,0) node[not] {} -- (0,1.5) node[xi] {};
\draw (-2,1) node[] {\tiny{i}};
\draw (-2,2.5) node[] {\tiny{j}};
\draw (2,1) node[] {\tiny{j}};
\draw (0,2.5) node[] {\tiny{i}};
 }

\DeclareSymbol{Xi4eabisc1quater}{1.5}{\draw (-1,2.5) node[xic] {} -- (-1,1) ; \draw[kernels2] (-1,1) node[xi] {} -- (0,0); 
 \draw (0,0)  -- (1,1) node[xic] {};%78
\draw[kernels2] (0,0) node[not] {} -- (0,1.5) node[xi] {};
 }

\DeclareSymbol{Xi4eabisc2}{1.5}{\draw (-1,2.5) node[xic] {} -- (-1,1) ; \draw[kernels2] (-1,1) node[xi] {} -- (0,0); %79
 \draw (0,0)  -- (1,1) node[xic] {};
\draw[kernels2] (0,0) node[not] {} -- (0,1.5) node[xi] {}; }

\DeclareSymbol{Xi4eabisc2l}{1.5}{\draw (-1,2.5) node[xiesf] {} -- (-1,1) ; \draw[kernels2] (-1,1) node[xi] {} -- (0,0); %79
 \draw (0,0)  -- (1,1) node[xic] {};
\draw[kernels2] (0,0) node[not] {} -- (0,1.5) node[xi] {}; }

\DeclareSymbol{Xi4eabisc2r}{1.5}{\draw (-1,2.5) node[xic] {} -- (-1,1) ; \draw[kernels2] (-1,1) node[xi] {} -- (0,0); %79
 \draw (0,0)  -- (1,1) node[xiesf] {};
\draw[kernels2] (0,0) node[not] {} -- (0,1.5) node[xi] {}; }

\DeclareSymbol{Xi4eabisc3}{1.5}{\draw (-1,2.5) node[xic] {} -- (-1,1) ; \draw[kernels2] (-1,1) node[xic] {} -- (0,0); %80
 \draw (0,0)  -- (1,1) node[xi] {};
\draw[kernels2] (0,0) node[not] {} -- (0,1.5) node[xi] {}; }

\DeclareSymbol{Xi4eb}{0}{%81
\draw[kernels2] (0,2) node[xi] {} -- (-1,1) ; \draw[kernels2] (-2,2)  node[xi] {} -- (-1,1) ; \draw (-1,1)  node[not] {} -- (0,0); 
 \draw (0,0) node[xi] {}  -- (1,1) node[xi] {};
}

\DeclareSymbol{Xi4eab}{1.5}{\draw[kernels2] (-1,2.5) node[xi] {} -- (-1,1) ; \draw[kernels2] (-2,2)  node[xi] {} -- (-1,1) ; \draw (-1,1)  node[not] {} -- (0,0); %82
 \draw[kernels2] (0,0)  -- (1,1) node[xi] {};
\draw[kernels2] (0,0) node[not] {} -- (0,1.5) node[xi] {}; 
}

\DeclareSymbol{Xi4eabdis}{1.5}{\draw[kernels2] (-1,2.5) node[xies] {} -- (-1,1) ; \draw[kernels2] (-2,2)  node[xies] {} -- (-1,1) ; \draw (-1,1)  node[not] {} -- (0,0); %82
 \draw[kernels2] (0,0)  -- (1,1) node[xies] {};
\draw[kernels2] (0,0) node[not] {} -- (0,1.5) node[xies] {}; 
}

\DeclareSymbol{Xi4eabc1}{1.5}{\draw[kernels2] (-1,2.5) node[xic] {} -- (-1,1) ; \draw[kernels2] (-2,2)  node[xi] {} -- (-1,1) ; \draw (-1,1)  node[not] {} -- (0,0); %83
 \draw[kernels2] (0,0)  -- (1,1) node[xic] {};
\draw[kernels2] (0,0) node[not] {} -- (0,1.5) node[xi] {}; 
}

\DeclareSymbol{Xi4eabc2}{1.5}{\draw[kernels2] (-1,2.5) node[xi] {} -- (-1,1) ; \draw[kernels2] (-2,2)  node[xi] {} -- (-1,1) ; \draw (-1,1)  node[not] {} -- (0,0); %84
 \draw[kernels2] (0,0)  -- (1,1) node[xic] {};
\draw[kernels2] (0,0) node[not] {} -- (0,1.5) node[xic] {}; 
}

\DeclareSymbol{Xi4eabbis}{1.5}{\draw[kernels2] (-1,2.5) node[xi] {} -- (-1,1) ; \draw[kernels2] (-2,2)  node[xi] {} -- (-1,1) ; \draw[kernels2] (-1,1)  node[not] {} -- (0,0); %85
 \draw (0,0)  -- (1,1) node[xi] {};
\draw[kernels2] (0,0) node[not] {} -- (0,1.5) node[xi] {}; 
}

\DeclareSymbol{Xi4eabbisc1}{1.5}{\draw[kernels2] (-1,2.5) node[xic] {} -- (-1,1) ; \draw[kernels2] (-2,2)  node[xi] {} -- (-1,1) ; \draw[kernels2] (-1,1)  node[not] {} -- (0,0); %86
 \draw (0,0)  -- (1,1) node[xic] {};
\draw[kernels2] (0,0) node[not] {} -- (0,1.5) node[xi] {}; 
}

\DeclareSymbol{Xi4eabbisc1perm}{1.5}{\draw[kernels2] (-1,2.5) node[xi] {} -- (-1,1) ; \draw[kernels2] (-2,2)  node[xic] {} -- (-1,1) ; \draw[kernels2] (-1,1)  node[not] {} -- (0,0); %86
 \draw (0,0)  -- (1,1) node[xic] {};
\draw[kernels2] (0,0) node[not] {} -- (0,1.5) node[xi] {}; 
}

\DeclareSymbol{Xi4eabbisc2}{1.5}{\draw[kernels2] (-1,2.5) node[xi] {} -- (-1,1) ; \draw[kernels2] (-2,2)  node[xi] {} -- (-1,1) ; \draw[kernels2] (-1,1)  node[not] {} -- (0,0); %87
 \draw (0,0)  -- (1,1) node[xic] {};
\draw[kernels2] (0,0) node[not] {} -- (0,1.5) node[xic] {}; 
}

\DeclareSymbol{Xi2cbis}{0}{\draw[kernels2] (0,1) -- (0.8,2.2) node[xi] {};\draw[kernels2] (0,-0.25) node[not] {} -- (0,1); \draw[kernels2] (0,1) node[not] {} -- (-0.8,2.2) node[xi] {};}%88

\DeclareSymbol{Xi2cbis1}{0}{\draw (0,1) -- (-0.8,2.2) node[xi] {};\draw[kernels2] (0,-0.25) node[not] {} -- (0,1) node[xi] {}; }%89

\DeclareSymbol{Xi2Xbis}{-2}{\draw[kernels2] (0,-0.25)  -- (-1,1) ; \draw (-1,1) node[xix] {};%91
\draw[kernels2] (0,-0.25) node[not] {} -- (1,1) node[xi] {};}

\DeclareSymbol{XXi2bis}{-2}{\draw[kernels2] (0,-0.25) -- (-1,1) node[xi] {};%92
\draw[kernels2] (0,-0.25) node[X] {} -- (1,1) node[xi] {};}

\DeclareSymbol{I1XiIXi}{0}{\draw[kernels2] (0,-0.25) -- (1,1) node[xi] {};%93
\draw (0,-0.25) node[not] {} -- (-1,1) node[xi] {};}

\DeclareSymbol{I1XiIXib}{0}{\draw  (0,-0.25) node[xi] {} -- (0,1) node[not] {};
\draw[kernels2] (0,1) -- (0,2.25) ; \draw (0,2.25) node[xi]{}; }%94

\DeclareSymbol{I1XiIXic}{0}{%95
\draw[kernels2] (0,0) -- (1,1) node[xi] {} ; 
\draw[kernels2] (0,0) node[not] {}  -- (-1,1) node[not] {} -- (0,2) node[xi] {};
}

\DeclareSymbol{thin}{1.4}{\draw[pagebackground] (-0.3,0) -- (0.3,0); \draw  (0,0) -- (0,2);}
\DeclareSymbol{thin2}{1.4}{\draw[pagebackground] (-0.3,0) -- (0.3,0); \draw[tinydots]  (0,0) -- (0,2);}

\DeclareSymbol{thick}{1.4}{\draw[pagebackground] (-0.3,0) -- (0.3,0); \draw[kernels2]  (0,0) -- (0,2);}

\DeclareSymbol{thick2}{1.4}{\draw[pagebackground] (-0.3,0) -- (0.3,0); \draw[kernels2,tinydots]  (0,0) -- (0,2);}

\DeclareSymbol{Xi4ind}{2}{\draw (0,0) node[xi,label={[label distance=-0.2em]right: \scriptsize  $ i $}]  { } -- (-1,1) node[xi,label={[label distance=-0.2em]left: \scriptsize  $ j $}] {} -- (0,2) node[xi,label={[label distance=-0.2em]right: \scriptsize  $ k $}] {} -- (-1,3) node[xi,label={[label distance=-0.2em]left: \scriptsize  $ \ell $}] {};}%115

\DeclareSymbol{Xi4c1}{2}{\draw (0,0) node[xic] {} -- (-1,1) node[xi] {} -- (0,2) node[xic] {} -- (-1,3) node[xi] {};} %119
\DeclareSymbol{IXi2ex}{0}{\draw (0,-0.25) node[xie] {} -- (-1,1) node[xi] {} ; \draw (0,-0.25)-- (1,1) node[xi] {};}
\DeclareSymbol{IXi2ex1}{0}{\draw (0,-0.25) node[xie] {} -- (-1,1) node[xi] {} -- (0,2) node[xi] {};}

\DeclareSymbol{Xi4b1}{0}{\draw(0,1.5) node[xic] {} -- (0,0); \draw (-1,1) node[xic] {} -- (0,0) node[xi] {} -- (1,1) node[xi] {};}

\DeclareSymbol{Xi4ec1}{0}{\draw (0,2) node[xi] {} -- (-1,1) node[xic] {} -- (0,0) node[xic] {} -- (1,1) node[xi] {};}
\DeclareSymbol{Xi4ec2}{0}{\draw (0,2) node[xic] {} -- (-1,1) node[xi] {} -- (0,0) node[xic] {} -- (1,1) node[xi] {};}
\DeclareSymbol{Xi4ec3}{0}{\draw (0,2) node[xic] {} -- (-1,1) node[xic] {} -- (0,0) node[xi] {} -- (1,1) node[xi] {};}

\DeclareSymbol{I1Xi4ac1}{2}{\draw[kernels2] (0,0) node[not] {} -- (-1,1) ; \draw[kernels2] (0,0) node[not] {} -- (1,1) node[xic] {} ;
\draw (-1,1) node[xi] {} -- (0,2) node[xic] {} -- (-1,3) node[xi] {};}%161

\DeclareSymbol{I1Xi4ac2}{2}{\draw[kernels2] (0,0) node[not] {} -- (-1,1) ; \draw[kernels2] (0,0) node[not] {} -- (1,1) node[xic] {} ;
\draw (-1,1) node[xi] {} -- (0,2) node[xi] {} -- (-1,3) node[xic] {};}

\DeclareSymbol{I1Xi4bp}{2}{\draw (0,0) node[not] {} -- (-1,1) node[xi] {} -- (0,2) ; \draw[kernels2] (0,2) node[not] {} -- (-1,3) node[xi] {};\draw[kernels2] (0,2)  -- (1,3) node[xi] {};
}

\DeclareSymbol{I1Xi4bc1}{2}{\draw (0,0) node[xic] {} -- (-1,1) node[xi] {} -- (0,2) ; \draw[kernels2] (0,2) node[not] {} -- (-1,3) node[xi] {};\draw[kernels2] (0,2)  -- (1,3) node[xic] {};
}%165

\DeclareSymbol{I1Xi4bc2}{2}{\draw (0,0) node[xic] {} -- (-1,1) node[xi] {} -- (0,2) ; \draw[kernels2] (0,2) node[not] {} -- (-1,3) node[xic] {};\draw[kernels2] (0,2)  -- (1,3) node[xi] {};
}

\DeclareSymbol{I1Xi4cp}{2}{\draw (0,0) node[not] {} -- (-1,1) node[not] {}; \draw[kernels2] (-1,1) -- (0,2) ; 
\draw[kernels2] (-1,1) -- (-2,2) node[xi] {} ;
\draw (0,2) node[xi] {} -- (-1,3) node[xi] {};}

\DeclareSymbol{I1Xi4cc1}{2}{\draw (0,0) node[xic] {} -- (-1,1) node[not] {}; \draw[kernels2] (-1,1) -- (0,2) ; 
\draw[kernels2] (-1,1) -- (-2,2) node[xi] {} ;
\draw (0,2) node[xic] {} -- (-1,3) node[xi] {};}%169

\DeclareSymbol{I1Xi4cc2}{2}{\draw (0,0) node[xic] {} -- (-1,1) node[not] {}; \draw[kernels2] (-1,1) -- (0,2) ; 
\draw[kernels2] (-1,1) -- (-2,2) node[xi] {} ;
\draw (0,2) node[xi] {} -- (-1,3) node[xic] {};}

\DeclareSymbol{I1Xi4abc1}{2}{\draw[kernels2] (0,0) node[not] {} -- (-1,1) ; \draw[kernels2] (0,0) node[not] {} -- (1,1) node[xic] {};\draw (-1,1) node[xi] {} -- (0,2) ; \draw[kernels2] (0,2) node[not] {} -- (-1,3) node[xic] {};\draw[kernels2] (0,2)  -- (1,3) node[xi] {}; }

\DeclareSymbol{I1Xi4abc2}{2}{\draw[kernels2] (0,0) node[not] {} -- (-1,1) ; \draw[kernels2] (0,0) node[not] {} -- (1,1) node[xic] {};\draw (-1,1) node[xi] {} -- (0,2) ; \draw[kernels2] (0,2) node[not] {} -- (-1,3) node[xi] {};\draw[kernels2] (0,2)  -- (1,3) node[xic] {}; }%173

\DeclareSymbol{R1}{0}{\draw (-1,1) node[xi] {} -- (0,0) node[not] {};%131
\draw[kernels2] (0,1.5) node[xic] {} -- (0,0) -- (1,1) node[xic] {};}
\DeclareSymbol{R2}{0}{\draw (-1,1) node[xic] {} -- (0,0) node[not] {};
\draw[kernels2] (0,1.5)  {} -- (0,0) -- (1,1)  {};
\draw (0,1.5) node[xi] {};
\draw (1,1) node[xic] {};
}%132
\DeclareSymbol{R3}{1}{\draw[kernels2] (-1,1.5)  {} -- (0,0) node[not] {} -- (1,1.5);%133
\draw (-1,1.5) node[xi] {};
\draw[kernels2] (0,3) {} -- (1,1.5) -- (2,3)  {};
\draw  (0,3) node[xic] {} ;
\draw (2,3) node[xic] {};}
\DeclareSymbol{R4}{1}{\draw[kernels2] (-1,1.5) node[xic] {} -- (0,0) node[not] {} -- (1,1.5);%134
\draw[kernels2] (0,3) {} -- (1,1.5) -- (2,3) node[xic] {};
\draw (0,3) node[xi] {};}

\DeclareSymbol{I1Xi4bcp}{2}{\draw (0,0) node[not] {} -- (-1,1) node[not] {}; \draw[kernels2] (-1,1) -- (0,2) ; %175
\draw[kernels2] (-1,1) -- (-2,2) node[xi] {} ; \draw[kernels2] (0,2) node[not] {} -- (-1,3) node[xi] {};\draw[kernels2] (0,2)  -- (1,3) node[xi] {};
}

\DeclareSymbol{I1Xi4bcc1}{2}{\draw (0,0) node[xic] {} -- (-1,1) node[not] {}; \draw[kernels2] (-1,1) -- (0,2) ; 
\draw[kernels2] (-1,1) -- (-2,2) node[xi] {} ; \draw[kernels2] (0,2) node[not] {} -- (-1,3) node[xi] {};\draw[kernels2] (0,2)  -- (1,3) node[xic] {};
}

\DeclareSymbol{I1Xi4bcc2}{2}{\draw (0,0) node[xic] {} -- (-1,1) node[not] {}; \draw[kernels2] (-1,1) -- (0,2) ; 
\draw[kernels2] (-1,1) -- (-2,2) node[xi] {} ; \draw[kernels2] (0,2) node[not] {} -- (-1,3) node[xic] {};\draw[kernels2] (0,2)  -- (1,3) node[xi] {};
} %177

\DeclareSymbol{2I1Xi4bc1}{2}{\draw[kernels2] (0,0) node[not] {} -- (-1,1) ;
\draw[kernels2] (0,0) -- (1,1);
\draw (-1,1) node[xic] {} -- (-1,2.5) node[xi] {};
\draw (1,1)  node[xic] {} -- (1,2.5) node[xi] {};
}

\DeclareSymbol{2I1Xi4bc2}{2}{\draw[kernels2] (0,0) node[not] {} -- (-1,1) ;
\draw[kernels2] (0,0) -- (1,1);
\draw (-1,1) node[xi] {} -- (-1,2.5) node[xic] {};
\draw (1,1)  node[xic] {} -- (1,2.5) node[xi] {};
}

\DeclareSymbol{diff2I1Xi4bc2}{2}{\draw (0,0) node[diff] {} -- (-1,1) ;
\draw (0,0) -- (1,1);
\draw (-1,1) node[xi] {} -- (-1,2.5) node[xic] {};
\draw (1,1)  node[xic] {} -- (1,2.5) node[xi] {};
}

\DeclareSymbol{2I1Xi4bc3}{2}{\draw[kernels2] (0,0) node[not] {} -- (-1,1) ;
\draw[kernels2] (0,0) -- (1,1);
\draw (-1,1) node[xic] {} -- (-1,2.5) node[xic] {};
\draw (1,1)  node[xi] {} -- (1,2.5) node[xi] {};
}

\DeclareSymbol{Xi41}{0}{\draw (0,1) -- (0.8,2.2) node[xic] {};\draw (0,-0.25) node[xi] {} -- (0,1) node[xi] {} -- (-0.8,2.2) node[xic] {};} 

\DeclareSymbol{Xi42}{0}{\draw (0,1) -- (0.8,2.2) node[xi] {};\draw (0,-0.25) node[xic] {} -- (0,1) node[xi] {} -- (-0.8,2.2) node[xic] {};}

\DeclareSymbol{Xi4ca1}{0}{\draw (0,1) -- (-1,2.2) node[xic] {};\draw (0,-0.25) node[xi] {} -- (0,1) ; \draw[kernels2] (0,1) node[not] {} -- (1,2.2) node[xic] {};
\draw[kernels2] (0,1) {} -- (0,2.7) node[xi] {};
}

\DeclareSymbol{Xi4ca2}{0}{\draw (0,1) -- (-1,2.2) node[xi] {};\draw (0,-0.25) node[xi] {} -- (0,1) ; \draw[kernels2] (0,1) node[not] {} -- (1,2.2) node[xic] {};
\draw[kernels2] (0,1) {} -- (0,2.7) node[xic] {};
}

\DeclareSymbol{Xi4cap}{0}{\draw (0,1) -- (-1,2.2) node[xi] {};\draw (0,-0.25) node[not] {} -- (0,1) ; \draw[kernels2] (0,1) node[not] {} -- (1,2.2) node[xi] {};
\draw[kernels2] (0,1) {} -- (0,2.7) node[xi] {};
}

\DeclareSymbol{Xi3a}{0}{
 \draw (-1,1)  node[xi] {} -- (0,0); 
 \draw (0,0) node[xi] {}  -- (1,1) node[xi] {};
}

\DeclareSymbol{Xi4ebc1}{0}{
\draw[kernels2] (0,2) node[xi] {} -- (-1,1) ; \draw[kernels2] (-2,2)  node[xic] {} -- (-1,1) ; \draw (-1,1)  node[not] {} -- (0,0); 
 \draw (0,0) node[xic] {}  -- (1,1) node[xi] {};
}

\DeclareSymbol{Xi4ebc2}{0}{
\draw[kernels2] (0,2) node[xi] {} -- (-1,1) ; \draw[kernels2] (-2,2)  node[xi] {} -- (-1,1) ; \draw (-1,1)  node[not] {} -- (0,0); 
 \draw (0,0) node[xic] {}  -- (1,1) node[xic] {};
}

\DeclareSymbol{Xi2cbispex}{0}{\draw[kernels2] (0,1) -- (0.8,2.2) node[xi] {};\draw (0,-0.25) node[xie] {} -- (0,1); \draw[kernels2] (0,1) node[not] {} -- (-0.8,2.2) node[xi] {};}

\DeclareSymbol{Xi2cbis1p}{0}{\draw (0,1) -- (-0.8,2.2) node[xi] {};\draw (0,-0.25) node[not] {} -- (0,1) node[xi] {}; }

\DeclareSymbol{Xi2Xp}{-2}{\draw (0,-0.25) node[not] {} -- (-1,1) node[xix] {};} % 229 not used

\DeclareSymbol{I1XiIXib}{0}{\draw  (0,-0.25) node[xi] {} -- (0,1) node[not] {};
\draw[kernels2] (0,1) -- (0,2.25) ; \draw (0,2.25) node[xi]{}; }

\DeclareSymbol{IXi2b}{0}{\draw  (0,-0.25) node[xi] {} -- (0,1) node[not] {};
\draw (0,1) -- (0,2.25) ; \draw (0,2.25) node[xi]{}; }

\DeclareSymbol{IXi2bex}{0}{\draw  (0,-0.25) node[xi] {} -- (0,1) node[xie] {};
\draw (0,1) -- (0,2.25) ; \draw (0,2.25) node[xi]{}; }

 \def\1{\mathbf{\symbol{1}}}

\DeclareSymbol{diff}{0}{
\draw (0,0.5) node[diff] {};
}

\DeclareSymbol{diff1}{0}{
\draw (0,0.5) node[diff1] {};
}

\DeclareSymbol{diff2}{0}{
\draw (0,0.5) node[diff2] {};
}

\DeclareSymbol{geo}{0}{
\draw (0,0) node[diff] {};
\draw (0.3,0) node[diff] {};
}

\DeclareSymbol{generic}{0}{
\draw (0,0.6) node[xi] {};
}

\DeclareSymbol{g}{0}{
\draw (0,0.6) node[g] {};
}

\DeclareSymbol{Ito}{0}{
\draw (0,0.6) node[xies] {};
}

\DeclareSymbol{Itob}{0}{
\draw (0,0.6) node[xiesf] {};
}

\DeclareSymbol{greycirc}{0}{
\draw (0,0.3) node[xi] {};
}

\DeclareSymbol{not}{0}{
\draw (0,0.6) node[not] {};
\draw[tinydots] (0,0.6) circle (0.8);
}

\DeclareSymbol{genericb}{0}{
\draw (0,0.6) node[xic] {};
}

\DeclareSymbol{bluecirc}{0}{
\draw (0,0.3) node[xic] {};
}

\DeclareSymbol{genericxix}{0}{
\draw (0,0.6) node[xix] {};
}

\DeclareSymbol{genericX}{0}{
\draw (0,0.6) node[X] {};
}

\DeclareSymbol{diffIto}{1}{
\draw  (0,2.5) -- (0,0) ;
\draw (0,-0.1) node[diff] {};
\draw (0,2.5) node[xies] {};
}
\DeclareSymbol{Itodiff}{2}{
\draw(0,2.9) -- (0,-0.2);
\draw (0,2.9) node[diff] {};
\draw (0,-0.1) node[xies] {};
}

\DeclareSymbol{diffgeneric}{1}{
\draw  (0,2.5) -- (0,0) ;
\draw (0,-0.1) node[diff] {};
\draw (0,2.5) node[xi] {};
}

\DeclareSymbol{genericdiff}{2}{
\draw(0,2.9) -- (0,-0.2);
\draw (0,2.9) node[diff] {};
\draw (0,-0.1) node[xi] {};
}

\DeclareSymbol{diffdot}{2}{
\draw  (0,3) -- (0,-0.1) ;
\draw (0,3) node[not] {};
\draw (0,-0.1) node[diff] {};
}

\DeclareSymbol{diffdotmini}{0}{
\draw  (0,0) -- (0,1.2) ;
\draw (0,1.2) node[not] {};
\draw (0,0) node[diffmini] {};
}

\DeclareSymbol{dotdiff}{2}{
\draw[kernelsmod]  (0,3) -- (0,-0.1) ;
\draw (0,3) node[diff] {};
\draw (0,-0.1) node[not] {};
}

\DeclareSymbol{dotdiff1}{2}{
\draw[kernelsmod]  (0,3) -- (0,-0.1) ;
\draw (0,3) node[diff1] {};
\draw (0,-0.1) node[not] {};
}

\DeclareSymbol{dotdiff1mini}{0}{
\draw[kernelsmod]  (0,1.2) -- (0,0) ;
\draw (0,1.2) node[diffmini] {};
\draw (0,0) node[not] {};
}

\DeclareSymbol{dotdiff2}{2}{
\draw (0,3) -- (0,-0.1) ;
\draw (0,3) node[diff] {};
\draw (0,-0.1) node[not] {};
}

\DeclareSymbol{dotdiff2mini}{0}{
\draw (0,1.2) -- (0,0) ;
\draw (0,1.2) node[diffmini] {};
\draw (0,0) node[not] {};
}

\DeclareSymbol{dotdiffstraight}{0}{
\draw  (0,3) -- (0,-0.1) ;
\draw (0,3) node[diff] {};
\draw (0,-0.1) node[not] {};
}

\DeclareSymbol{arbre1}{0}{
\draw  (0,0) -- (1.5,1.5) ;
\draw (1.5,1.5) node[not] {};
\draw (0,0) node[not] {};
}

\DeclareSymbol{arbre2}{0}{
\draw  (0,0) -- (1.5,1.5) ;
\draw[kernelsmod] (0,0) -- (-1.5,1.5);
\draw (1.5,1.5) node[not] {};
\draw (0,0) node[not] {};
\draw (-1.5,1.5) node[xi] {};
}

\DeclareSymbol{arbre3}{0}{
\draw  (0,0) -- (1.5,1.5) ;
\draw[kernelsmod] (1.5,1.5) -- (0,3);
\draw (0,0) node[not] {};
\draw (1.5,1.5) node[not] {};
\draw (0,3) node[xi] {};
}

\DeclareSymbol{treeeval}{0}{
\draw (0,0) -- (1,1);
\draw (0,0) node[xi] {};
\draw (1.25,1.25) node[xi] {};
\draw (-0.6,0.6) node[]{\tiny{$i$}};
\draw (0.65,1.85) node[]{\tiny{$j$}};
}

\DeclareSymbol{testeval}{0}{
\draw (0,0) -- (1,1);
\draw (0,0) -- (-1,1);
\draw (0,0) node[xi] {};
\draw (1.25,1.25) node[xi] {};
\draw (-1.25,1.25) node[xi] {};
\draw (-0.6,-0.6) node[]{\tiny{$i$}};
\draw (0.65,1.85) node[]{\tiny{$j$}};
\draw (-1.95,1.85) node[]{\tiny{$k$}};
}

\DeclareSymbol{treeeval2}{0}{
\draw[kernelsmod] (-0.25,-1) -- (1,0.5) ;
\draw[kernelsmod] (1,0.5) -- (-0.25,2);
\draw (1,0.5) node[diff2] {};
\draw (-0.25,-1) node[not] {};
\draw (-0.25,2) node[xi] {};
\draw (-0.6,1.2) node[]{\tiny{1}};
}

\DeclareSymbol{arbreact}{1}{
\draw (0,0) node[not] {};
\draw[kernelsmod] (0,0) -- (1,1);
\draw[kernelsmod] (0,0) -- (-1,1);
\draw (-1,1) node[xic] {};
\draw  (0,2) -- (1,1) ;
\draw (0,2) node[xic] {};
\draw (1,1) node[xi] {};
}

\DeclareSymbol{arbreact1}{0}{
\draw (0,-1.5) -- (0,0);
\draw[kernelsmod] (0,0) -- (1,1);
\draw[kernelsmod] (0,0) -- (-1,1);
\draw  (0,2) -- (1,1) ;
\draw (0,-1.5) node[diff] {};
\draw (0,0) node[not] {};
\draw (-1,1) node[xic] {};
\draw (0,2) node[xic] {};
\draw (1,1) node[xi] {};
}

\DeclareSymbol{arbreact2}{0}{
\draw (0,-0.75) -- (-1,0.5); 
\draw (0,-0.75) -- (1,0.5);
\draw (0,1.5) -- (1,0.5);
\draw (0,1.5) node[xic] {};
\draw (1,0.5) node[xi] {};
\draw (-1,0.5) node[xic] {};
\draw (0,-0.75) node[diff] {};
}

\DeclareSymbol{arbreact3}{0}{
\draw[kernelsmod] (0,-0.75) -- (-1,0.5); 
\draw[kernelsmod] (0,-0.75) -- (1,0.5);
\draw (0,1.75) -- (1,0.5);
\draw (2,1.75) -- (1,0.5);
\draw (0,1.75) node[xic] {};
\draw (1,0.5) node[diff] {};
\draw (-1,0.5) node[xic] {};
\draw (2,1.75) node[xi] {};
\draw (0,-0.75) node[not] {};
}

\DeclareSymbol{pre_im_I1Xitwo}{0}{
\draw[kernels2] (0,-0.3) node[not] {} -- (-0.6,0.7) ;
\draw[kernels2] (0,-0.3) -- (0.6,0.7);
\draw (0,0.9) node[g] {};
}

\DeclareSymbol{pre_im_cI1Xi4ab}{2}{
\draw[kernels2] (0,-1) node[not] {} -- (-0.6,0) ;
\draw[kernels2] (0,-1) -- (0.6,0);
\draw (0,0.2) node[g] {};
\draw (0,0.6) -- (0,1.5);
\draw[kernels2] (0,1.5) node[not] {} -- (-0.6,2.5) ;
\draw[kernels2] (0,1.5) -- (0.6,2.5);
\draw (0,2.7) node[g] {};
}%46

\DeclareSymbol{pre_im_I1Xi4acc2}{0}{
\draw[kernels2] (-1,-0.5) node[not] {} -- (-1.6,0.5) ;
\draw[kernels2] (-1,-0.5) -- (-0.4,0.5);
\draw[kernels2] (-1,-0.5) -- (0.2,-1.5) node[not] {} ;
\draw (-1,1.1) -- (-1,2);
\draw[kernels2] (0.2,-1.5) -- (0.2,2);
\draw (-1,0.7) node[g] {};
\draw (-0.3,2.2) node[g] {};
}

\DeclareSymbol{pre_im_I1Xi4abcc2}{2}{
\draw[kernels2] (0,-1) node[not] {} -- (-1,0) node[not] {};
\draw[kernels2] (-1,1.2) node[not] {} -- (-1,0);
\draw[kernels2] (-1,1.2) -- (-1.5,2.5);
\draw[kernels2] (-1,1.2) -- (-0.5,2.5);
\draw[kernels2] (-1,0) -- (0.7,2.5);
\draw[kernels2] (0,-1) -- (1.5,2.5);
\draw (-1,2.7) node[g] {};
\draw (1,2.7) node[g] {};
}

\DeclareSymbol{pre_im_2I1Xi4c1}{2}{
\draw[kernels2] (0,-0.5) node[not] {} -- (-1,0.5) node[not] {};
\draw[kernels2] (0,-0.5) -- (1,0.5) node[not] {};
\draw[kernels2] (-1,0.5) node[not] {}-- (-1.7,2);
\draw[kernels2]  (-1,2) -- (1,0.5);
\draw[kernels2] (-1,0.5) -- (1,2);
\draw[kernels2] (1,0.5) -- (1.7,2);
\draw (-1.2,2.2) node[g] {};
\draw (1.2,2.2) node[g] {};
}

\DeclareSymbol{pre_im_Xi4eabisc2}{2}{
\draw[kernels2] (1.2,-0.5) node[not] {} -- (-0.7,0.8) ;
\draw[kernels2] (1.2,-0.5) -- (0.4,0.8);
\draw (0,1.4)  -- (0,2.2);
\draw (1.2,2.2) -- (1.2,-0.6);
\draw (0,1) node[g] {};
\draw (0.6,2.4) node[g] {};
}

\DeclareSymbol{pre_im_Xi4eabisc22}{2}{
\draw (1.2,-0.5) node[not] {} -- (-0.7,0.8) ;
\draw[kernels2] (1.2,-0.5) -- (0.4,0.8);
\draw (0,1.4)  -- (0,2.2);
\draw[kernels2] (1.2,2.2) -- (1.2,-0.6);
\draw (0,1) node[g] {};
\draw (0.6,2.4) node[g] {};
}

\DeclareSymbol{pre_im_Xi4eabisc222}{2}{
\draw[kernels2] (0.4,-0.5) node[not] {} -- (-0.6,1) ;
\draw[kernels2] (1.2,0) -- (0.3,1);
\draw (0,1.1)  -- (0,2.5);
\draw[kernels2] (1.2,2.5) -- (1.2,0) node[not] {} -- (0.4,-0.6);
\draw (0,1.2) node[g] {};
\draw (0.6,2.5) node[g] {};
}

\DeclareSymbol{pre_im_Xi4eabc2}{2}{
\draw (0,-0.5) node[not] {} -- (-1,0.5) node[not] {};
\draw[kernels2] (-1,0.5) -- (-1.5,2);
\draw[kernels2] (0,-0.5)  -- (0.7,2);
\draw[kernels2] (-1,0.5) -- (-0.5,2);
\draw[kernels2] (0,-0.5) -- (1.5,2);
\draw (-1,2.2) node[g] {};
\draw (1,2.2) node[g] {};
}

\DeclareSymbol{pre_im_Xi4eabbisc2}{2}{
\draw[kernels2] (0,-0.5) node[not] {} -- (-1,0.5) node[not] {};
\draw[kernels2] (-1,0.5) -- (-1.5,2);
\draw[kernels2] (0,-0.5)  -- (0.7,2);
\draw[kernels2] (-1,0.5) -- (-0.5,2);
\draw (0,-0.5) -- (1.5,2);
\draw (-1,2.2) node[g] {};
\draw (1,2.2) node[g] {};
}

\DeclareSymbol{pre_im_I1Xi4abcc1}{2}{
\draw[kernels2] (0,-1) node[not] {} -- (-1,0) node[not] {};
\draw[kernels2] (0,1.1) node[not] {} -- (-1,0);
\draw[kernels2] (-1,0) -- (-1.5,2.5);
\draw[kernels2] (0,1.1) node[not] {} -- (-0.5,2.5);
\draw[kernels2] (0,1.1) -- (0.5,2.5);
\draw[kernels2] (0,-1) -- (1.5,2.5);
\draw (-1,2.7) node[g] {};
\draw (1,2.7) node[g] {};
}

\DeclareSymbol{pre_im_Xi4eabc1}{2}{
\draw (0,-0.5) node[not] {} -- (-1,0.5) node[not] {};
\draw[kernels2] (-1,0.5) -- (-1.5,2);
\draw[kernels2] (0,-0.5)  -- (-0.5,2);
\draw[kernels2] (-1,0.5) -- (0.8,2);
\draw[kernels2] (0,-0.5) -- (1.5,2);
\draw (-1,2.2) node[g] {};
\draw (1,2.2) node[g] {};
}

\DeclareSymbol{pre_im_Xi4ba1b}{2}{
\draw[kernels2] (0,0) node[not] {}  -- (1.8,1.5);
\draw[kernels2] (0,0) -- (0.8,1.5);
\draw (0,-0.1) -- (-1.8,1.5);
\draw (0,-0.1) -- (-0.8,1.5);
\draw (-1,1.7) node[g] {};
\draw (1,1.7) node[g] {};
}

\DeclareSymbol{pre_im_Xi4ba2}{2}{
\draw (0,-0.1) node[not] {}  -- (1.8,1.5);
\draw[kernels2] (0,0) -- (0.8,1.5);
\draw (0,-0.1) -- (-1.8,1.5);
\draw[kernels2] (0,0) -- (-0.8,1.5);
\draw (-1,1.7) node[g] {};
\draw (1,1.7) node[g] {};
}

\DeclareSymbol{pre_im_Xi4cabc2}{2}{
\draw[kernels2] (0,-0.5) node[not] {} -- (-1,0.5) node[not] {};
\draw[kernels2] (-1,0.5) -- (-1.5,2);
\draw (-1,0.5)  -- (0.7,2);
\draw[kernels2] (-1,0.5) -- (-0.5,2);
\draw[kernels2] (0,-0.5) -- (1.5,2);
\draw (-1,2.2) node[g] {};
\draw (1,2.2) node[g] {};
}

\DeclareSymbol{pre_im_Xi4cabc1}{2}{
\draw[kernels2] (0,-0.5) node[not] {} -- (-1,0.5) node[not] {};
\draw (-1,0.5) -- (-1.5,2);
\draw[kernels2] (-1,0.5)  -- (0.7,2);
\draw[kernels2] (-1,0.5) -- (-0.5,2);
\draw[kernels2] (0,-0.5) -- (1.5,2);
\draw (-1,2.2) node[g] {};
\draw (1,2.2) node[g] {};
}

\DeclareSymbol{pre_im_Xi4eabbisc1}{2}{
\draw[kernels2] (0,-0.5) node[not] {} -- (-1,0.5) node[not] {};
\draw[kernels2] (-1,0.5) -- (-1.5,2);
\draw[kernels2] (0,-0.5)  -- (-0.5,2);
\draw[kernels2] (-1,0.5) -- (0.8,2);
\draw (0,-0.5) -- (1.5,2);
\draw (-1,2.2) node[g] {};
\draw (1,2.2) node[g] {};
}

\DeclareSymbol{pre_im_1}{0}{
\draw[kernels2] (0,-0.5) node[not] {} -- (-0.6,0.5) ;
\draw[kernels2] (0,-0.5) -- (0.6,0.5);
\draw (0,1.1)  -- (-0.55,2);
\draw (0,1.1)  -- (0.55,2);
\draw (0,0.7) node[g] {};
\draw (0,2.2) node[g] {};
}

\DeclareSymbol{disconnect}{0}{
\draw[kernels2] (0,-0.5) node[not] {} -- (-0.6,0.5) ;
\draw[kernels2] (0,-0.5) -- (0.6,0.5);
\draw (-0.55,1.1)  -- (-0.55,2.3);
\draw (0.55,2.3) -- (0.55,1.5) -- (1.2,1.5) -- (1.2,3.5) -- (0.55,3.5) -- (0.55,2.7);
\draw (0,0.7) node[g] {};
\draw (0,2.5) node[g] {};
}

\DeclareSymbol{pre_im_2}{2}{\draw[kernels2] (0,0) node[not] {} -- (-1,1) node[not] {};
\draw[kernels2] (0,0) -- (1,1) node[not] {};
\draw[kernels2] (-1,1) -- (-1.5,2.5);
\draw[kernels2] (-1,1) -- (-0.5,2.5);
\draw[kernels2] (1,1) -- (0.5,2.5);
\draw[kernels2] (1,1) -- (1.5,2.5);
\draw (-1,2.7) node[g] {};
\draw (1,2.7) node[g] {};
}

\DeclareSymbol{CX_rec}{0}{
\draw [black] (-0.3,1) to (-0.3,-0.3);
\draw [black] (0.3,1) to (0.3,-0.3);
\draw [black] (-0.3,1) to (-0.3,2.3);
\draw [black] (0.3,1) to (0.3,2.3);
\draw (0,1) node[rec] {};
}

\DeclareSymbol{CX_cerc}{0}{
\draw [black] (0,1) to (0,-0.3);
\draw (0,1) node[cerc] {};
}

\DeclareMathAlphabet{\mathpzc}{OT1}{pzc}{m}{it}

\def\eqref#1{(\ref{#1})}

\makeatletter % Stolen from the internet to make a fat \cdot which isn't as fat as a \bullet
\newcommand*{\bigcdot}{}% Check if undefined
\DeclareRobustCommand*{\bigcdot}{%
  \mathbin{\mathpalette\bigcdot@{}}%
}
\newcommand*{\bigcdot@scalefactor}{.5}
\newcommand*{\bigcdot@widthfactor}{1.15}
\newcommand*{\bigcdot@}[2]{%
  \sbox0{$#1\vcenter{}$}% math axis
  \sbox2{$#1\cdot\m@th$}%
  \hbox to \bigcdot@widthfactor\wd2{%
    \hfil
    \raise\ht0\hbox{%
      \scalebox{\bigcdot@scalefactor}{%
        \lower\ht0\hbox{$#1\bullet\m@th$}%
      }%
    }%
    \hfil
  }%
}
\makeatother

\tcbset
{colframe=boxcolor,colback=symbols!7!pagebackground,coltext=pageforeground,
fonttitle=\bfseries,nobeforeafter,center title,size=fbox,boxsep=1.5pt,
top=0mm,bottom=0mm,boxsep=0mm,tcbox raise base}

\def\two{{\<generic>\kern0.05em\<genericb>}}
\def\twoI{{\<Ito>\kern0.05em\<Itob>}}

\def\st{\mathsf{fgt}}
\def\mail#1{\burlalt{#1}{mailto:#1}}

\usepackage{thmtools} %Added%

\begin{document}

\def\st{\mathsf{fgt}}
\def\mail#1{\burlalt{#1}{mailto:#1}}

\title{Novikov algebras and multi-indices in regularity structures}
\author{Yvain Bruned$^1$, Vladimir Dotsenko$^2$}
\institute{ 
 IECL (UMR 7502), Université de Lorraine
 \and IRMA (UMR 7501), Université de Strasbourg \\
Email:\ \begin{minipage}[t]{\linewidth}
\mail{yvain.bruned@univ-lorraine.fr},
\\ \mail{vdotsenko@unistra.fr}.
\end{minipage}}
%Added by Foivos - Shuffle Product symbol
\def\dsqcup{\sqcup\mathchoice{\mkern-7mu}{\mkern-7mu}{\mkern-3.2mu}{\mkern-3.8mu}\sqcup}
%Added by Foivos - Shuffle Product symbol

\maketitle

\begin{abstract}
\ \ \ \ In this work, we introduce multi-Novikov algebras, a generalisation of Novikov algebras with several binary operations indexed by a given set, and show that the multi-indices recently introduced in the context of singular stochastic partial differential equations can be interpreted as free multi-Novikov algebras. This is parallel to the fact that decorated rooted trees arising in the context of regularity structures are related to free multi-pre-Lie algebras.  
\end{abstract}

\setcounter{tocdepth}{1}
\tableofcontents

\section{ Introduction }

The study of singular stochastic partial differential equations (SPDEs) has reached a high degree of generality via the theory of regularity structures invented by Martin Hairer in \cite{reg}. One of the key ideas in this domain is to provide a local description of the solution using Butcher-type series based on decorated trees. These series are made of recentered iterated integrals built upon the noise, the differential operator of the equation and its non-linearities. Then, Hopf algebras describe the recentering of these integrals and their renormalisation as they contain the ill-defined distributional products from the singular SPDEs. The main step on the algebraic side in \cite{BHZ} was to introduce the concept of decorated trees with decorations both on the nodes and the edges. Two Hopf algebras in co-interaction are built for constructing the renormalised model that are the renormalised recentered iterated integrals describing locally the solution of a given singular SPDE (Hopf algebras in (co)interaction were formalised in the work \cite{CEFM11}. Then, \cite{CH16} shows convergence of the renormalised model. Finally, \cite{BCCH} provides a systematic way to obtain the renormalised equation: renormalising the iterated integrals induces a change of the solution and therefore of the equation one starts with. For a review on these results see \cite{FrizHai,BaiHos}. The Hopf algebras at play in the context of singular SPDEs are nowadays sufficiently well understood. Indeed, one central object is the deformed grafting product \cite{BCCH,BM22} that can be extended to a post-Lie product (see  \cite{BK23}). This product is the cornerstone of the construction of the recentering Hopf algebra which allows to build recentered iterated integrals. This construction is then the result of the application of an adaptation of the Guin--Oudom construction \cite{Guin1,Guin2} to the post-Lie context, see \cite{ELM}. Post-Lie algebras have been first developed for posets partition in \cite{Val} and for Lie-Butcher integrators in \cite{ML08,ML13}.
A nice algebraic feature of the deformed grafting product is that it freely generates the space of decorated trees; this is an extension of the classical result on free pre-Lie algebras \cite{ChaLiv}. One of the consequences of that fact is that the pre-Lie morphism maps used for renormalisation are uniquely defined. This has been first used at the level of rough paths in \cite{BCFP} and then in the context of regularity structures in \cite{BCCH}.

Recently, a different combinatorial approach has been proposed for singular SPDEs in the context of regularity structures. The idea is to gather decorated trees that have the same coefficients in the  expansion of the solution into a unique symbol. In the simplest possible case, these symbols are multi-indices encoding the number of nodes of a given tree and for each node its arity. The multi-indices appear first in the context of quasi-linear equations in \cite{OSSW} and their Hopf algebra structures for the recentering has been unveiled in  \cite{LOT}. Then, convergence of the multi-indices model has been obtained in \cite{LOTT} see also \cite{T,GT} for further extensions. The renormalised equations and multi-indices models for a large class of equations are given in \cite{BL23}. One can look at \cite{LO23,OST} for surveys on multi-indices. The  approach of \cite{LOTT} is recursive and it is not using diagrams like in \cite{CH16}. Such an approach is also valid within the context of decorated trees (see \cite{BN23,HS,BH23}).
Moreover, the recentering Hopf algebra for multi-indices can be obtained via a post-Lie product introduced in \cite{BK23} (see also \cite{JZ} for an extension on the algebra of derivations). 

In the context of multi-indices, the corresponding pre-Lie algebra is not free, and therefore the fundamental free object has been missing since the introduction of these combinatorial objects.
The simplest possible instance of multi-indices corresponds to considering a set of abstract variables $(z_k)_{k \in \mathbb{N}}$, where the variable $z_k$ encodes nodes of the tree that have $k$ children. Multi-indices $\beta$ over $\mathbb{N}$  measure the frequency of the variables, and can be compactly represented as monomials 
\begin{equs}
	z^{\beta} : = \prod_{k \in \mathbb{N}} z_k^{\beta(k)}. 
\end{equs}
The pre-Lie product on the vector space of such monomials is defined as 
\begin{equs}
	z^{\beta}\triangleright  z^{\beta'}= z^{\beta}D(z^{\beta'}),
\end{equs}
where $D$ is the derivation given by
\begin{equs}
	D = \sum_{k \in \mathbb{N}} (k+1) z_{k+1} \partial_{z_k}.  
\end{equs}
The action of this operator corresponds to adding one child to one of the nodes of our tree in all possible ways. 
 This corresponds to the convention where the monomial corresponding to the tree comes with a combinatorial coefficient equal to the product of $\frac{1}{n_v!}$, where $n_v$ is the number of children of the node $v$. This way, adding a child to a node requires a correction of the coefficient given precisely by multiplication by $k+1$. This is more in agreement with the way B-series are presented, and corresponds to the convention adopted in \cite{LOT}. There is an alternative convention to~\cite{BHE25} where $k+1$ is replaced by $1$, which later has an impact on the symmetry factors of multi-indices in expansions/approximations of solutions of SPDEs/ODEs.

Since we would like to only consider the multi-indices that can be associated to trees, one is forced to focus on the multi-indices satisfying the so called ``populated'' condition \cite{LOT} postulating that
\begin{equs}
	\label{populated_1}
	 \sum_{k \in \mathbb{N}} (1 - k)\beta(k)  = 1.
\end{equs}
It was conjectured by Dominique Manchon that populated multi-indices form  the free Novikov algebra.
A Novikov algebra is a vector space equipped with a bilinear product $x,y\mapsto x\triangleright y$, satisfying the identities
\begin{gather}
  (x \triangleright y) \triangleright z  - x \triangleright (y \triangleright z) = (y \triangleright x) \triangleright z  - y \triangleright (x \triangleright z) 
  \\   (x \triangleright y) \triangleright z   = (x \triangleright z)\triangleright y ;
\end{gather}
in other words, in addition to the pre-Lie identity, the operators $r_y\colon x\mapsto x\triangleright y$ commute between themselves: $r_yr_z=r_zr_y$ for all $y,z$. This type of algebras was first considered in the study of Hamiltonian operators in the formal calculus of variations in \cite{GD}, and then rediscovered by Balinskii and Novikov in the context of classification of linear Poisson brackets of hydrodynamical type \cite{BN85}; the term ``Novikov algebras'' was introduced in \cite{O92}. It turns out that the corresponding theorem does exist in the literature; it goes back to \cite{DL}. The correspondence has been made explicit in \cite[Lemma 3.5]{Li23}\footnote{More precisely, we refer to the second arxiv version of \cite{Li23} that appeared around one week before the present work.}. 

\begin{theorem}\label{free_Novikov}{{\rm \cite[Lemma 3.5]{Li23}}} 
The Novikov algebra of populated multi-indices is isomorphic to the free algebra on one generator. 
\end{theorem}

In this paper, we prove a generalisation of this theorem that is suitable for singular SPDEs, as the theorem as it stands only applies to rough paths related to ordinary differential equations \cite{Li23}. General multi-indices are now defined using formal variables of the form $ z_{(\mfl,w)} $ where $\mfl$ belongs to a finite set $\mfL^-$ (depending on our SPDE) and $w$ is a commutative monomial in the alphabet $A=\mathbb{N}^{d+1}$. One can define a collection of derivations $ D^{(\bn)} $ indexed by $A$, and operations $\triangleright_a$, $a\in A$, and give appropriate populated conditions ensuring that the multi-indices arise from decorated trees. These very general multi-indices have been proposed in \cite{BL23}. Let us  introduce a new algebraic structure that we shall call the multi-Novikov algebra, for which such multi-indices form a free object.  Formally, a $A$-multi-Novikov algebra is a vector space equipped with bilinear products $x,y\mapsto x\triangleright_a y$ indexed by a set $ A $ which satisfy the identities
\begin{gather}
	(x \triangleright_a y) \triangleright_b z  - x \triangleright_a (y \triangleright_b z) = (y \triangleright_a x) \triangleright_b z  - y \triangleright_a (x \triangleright_b z)\label{eq:exchgen_bis}, 
	\\ (x \triangleright_a y) \triangleright_b z  - x \triangleright_a (y \triangleright_b z) = (x\triangleright_b y) \triangleright_a z - x \triangleright_b (y \triangleright_a z)\label{eq:exchA_bis},
	\\   (x \triangleright_a y) \triangleright_b z   = (x \triangleright_b z)\triangleright_a y\label{eq:rcom_bis},
\end{gather}
for all $ a,b\in A$. This is a generalisation of Novikov algebras which is an analogue of the generalisation from pre-Lie algebras to multi-pre-Lie algebras in \cite{BCCH}\footnote{{Shortly after this paper was put on the arxiv, the authors in \cite[Ex. 2.23]{ZGG23} recovered the identities of a multi-Novikov algebra from an operadic context.}}. 

Our first main result is the following generalisation of Theorem~\ref{free_Novikov}. 
\begin{theorem} \label{free_Novikov_general} 
The $\mathbb{N}^{d+1}$-multi-Novikov algebra of populated general multi-indices is isomorphic to free algebra generated by the set $\mfL^-$.
\end{theorem}

For capturing the complexity of the multi-indices for singular SPDEs, one has to introduce other derivations $\partial_i$, $0\le i\le d$, that satisfy, together with the derivations $ D^{(\bn)} $, the following relations:
\begin{equs}
  D^{(\bn)}D^{(\mathbf{m})} & =  D^{(\mathbf{m})}D^{(\bn)}, \quad \partial_i \partial_j = \partial_j \partial_i
  \\  D^{(\bn)}  \partial_i &  = \partial_i D^{(\bn)}  + n_iD^{(\bn-e_i)},
\end{equs}
where $e_i$ is the standard basis vector of $\mathbb{N}^{d+1}$.There is a corresponding generalisation of multi-indices which we shall call SPDE multi-indices. Our second main result is the following theorem.

\begin{theorem} \label{free_Novikov_general_monomial} 
The $\mathbb{N}^{d+1}$-multi-Novikov algebra of populated SPDE multi-indices is isomorphic to free algebra generated by the set $\mathbb{N}^{d+1}\times \mfL^-$.
\end{theorem}

Our approach also unravels the functoriality of such results: our theorems are directly compatible with the corresponding theorems on the multi-pre-Lie level.

Let us outline the paper by summarising the content of its sections. In Section~\ref{sec::2}, we introduce general multi-indices \eqref{general_multi_indices} associated to a given singular SPDE \eqref{set01}. These multi-indices were first introduced in \cite{BL23}.  They encode derivatives of the non-linearities in the solution with respect to the unknown function $u$ and its iterated partial derivatives. We recall the associated derivations \eqref{derivations} to these multi-indices together with the populated condition \eqref{populated_bis}. From Proposition~\ref{multi_Nov_example}, we motivate a new structure that we call multi-Novikov algebra in Definition~\ref{def_multi_novikov}. One of the main examples is the linear span of populated general multi-indices denoted by  $\mathcal{M}^{\mathrm{gen}}$  equipped with the products $ (\triangleright_{{\bf{n}}})_{{\bf{n}} \in \mathbb{N}^{d+1}}  $ defined in \eqref{def_products}. We finish the section by stating one of our main theorems that is Theorem~\ref{main_theorem_multi_indices} saying that the multi-Novikov algebra of populated general multi-indices is free.

In Section~\ref{sec::3}, for proving Theorem~\ref{main_theorem_multi_indices} we have to prove Theorem~\ref{th:freeNovikov} which is a non trivial extension of \cite[Theorem~7.8]{DL}. This theorem identifies the free multi-Novikov algebra with the multi-Novikov subalgebra of the free commutative multidifferential algebra generated by the generators of that latter algebra. This subalgebra can be identified as populated general multi-indices in Proposition~\ref{prop:lowerbound} and in the proof of Theorem \ref{main_theorem_multi_indices} at the end of the section. For proving Theorem~\ref{th:freeNovikov}, one starts with a lower bound in Corollary~\ref{lower_bound} on the dimension of the degree $n$ component of the free multi-Novikov algebra and then compute the same upper bound which allows to conclude. The strategy for the upper bound  is to consider the multi-magmatic algebra quotiented by the relations of multi-Novikov and to show that in Proposition~\ref{normal_coset} this object
is spanned by cosets of normal ordered left-leaning products introduced in Definition~\ref{leaning}  (leaning), Definition~\ref{ordered} (ordered) and Definition~\ref{normal} (normal). 

In Section~\ref{sec::4}, we introduce SPDE multi-indices which are based on words satisfying the relations \eqref{relation_words}. This is a new structure in comparison to the existing literature \cite{OSSW,LOT,BL23}. We explain the differences in Remark~\ref{remark_z_k}: it avoids the use of extra variables $z_{k}$ with $k \in \mathbb{N}^{d+1}$. We extend naturally the definition of the derivations and the populated condition. 

In Section~\ref{sec::5}, we start by observing in  Proposition~\ref{linearised_Novikov} that linearised versions of multi-pre-Lie algebras and multi-Novikov are still in the same class. Then, if $V=\mathrm{Vect}(A)$ ($ A $ is the index set for the products used) carries a representation of a Lie algebra $\mathfrak{g}$, then  one can set in   Definition~\ref{def_action_g} an action of $\mathfrak{g}$ as generalised derivations. One can characterise the free object for this new structure in
Proposition \ref{prop:extlin}.
The end of the section uses the previous proposition to prove Theorem~\ref{second_main_theorem_multi_indices} asserting that the multi-Novikov algebra of populated SPDE multi-indices is free.

In Section~\ref{sec::6}, we present the planar decorated trees quotiented by the relations \eqref{relation_trees} corresponding to the SPDE multi-indices. One can realise by functoriality that Theorem \ref{main_theorem_trees} like Theorem~\ref{second_main_theorem_multi_indices}  is also a consequence of Proposition\ref{prop:extlin}. We finish the section with the introduction of multi-symmetric braces given in Definition~\ref{multi_braces}. Their explicit expression for multi-Novikov structure in Proposition~\ref{brace_Novikov}
 allows us to recover in a conceptual way the morphism from decorated trees to multi-indices of \cite[Remark 5.6]{BK23}.

\subsection*{Acknowledgements}

{\small
  Y. B. thanks Dominique Manchon who conjectured a relationship of multi-indices to free Novikov algebras. Y. B. thanks the ANR via the project LoRDeT (Dynamiques de faible régularité via les arbres décorés) from the projects call T-ERC\_STG. The project LoRDeT allowed to fund a workshop in Nancy entitled "Hopf algebras, operads, deformations for singular dynamics" where the authors first met. Y. B. gratefully acknowledges funding support from the European Research Council (ERC) through the ERC Starting Grant Low Regularity Dynamics via Decorated Trees (LoRDeT), grant agreement No.\ 101075208. V. D. is funded by the ANR project HighAGT (ANR-20-CE40-0016), and by Institut Universitaire de France.
}

\section{General multi-indices}\label{sec::2}

In applications of the algebraic formalism developed in this paper, one deals with stochastic PDEs of the form
\begin{equation}\label{set01}
\left( \partial_t - 	\mcL \right) u = \sum_{\mfl\in\mfL^-} a^\mfl (\mathbf{u}) \xi_\mfl,
\end{equation}
where $u$ is a function of $d+1$ variables $t=x_0$, $x_1$, \ldots, $x_d$, $ \mathcal{L} $ is a differential operator in the variables $x_1$, \ldots, $x_d$, $\mfL^- $ is a finite set and the $ \xi_{\mfl} $ are space-time noises (random distributions). The bold letter $\mathbf{u}$ in each $a^\mfl(\mathbf{u})$ means that these are functions of $u$ and of its iterated partial derivatives.   For $\bn=(n_0,\ldots,n_d)\in\mathbb{N}^{d+1}$, one may consider the expressions 
 \[
u^{(\bn)}:=\frac{\partial_{x_0}^{n_0}\cdots \partial_{x_d}^{n_d}}{n_0!\cdots n_d!}(u),
 \]
which are the coefficients of the Taylor expansion of $u$ around the point $x$. To each $a^\mfl (\mathbf{u})$ one may apply the derivatives $\partial_{u^{(\bn)}}$ which pairwise commute. In this context, one typically works with polynomials in the expressions 
 \[
\left(\prod_{{\bf{n}}\in\mathbb{N}^{d+1}} \partial_{u^{(\bn)}}^{w_{\mfl}(\bn)}\right) 
(a^\mfl (\mathbf{u})),
 \]
where $w_{\mfl}(\bn)\in\mathbb{N}$ are the orders of derivatives applied to $a^\mfl (\mathbf{u})$. 
These expressions will be now encoded in a formal algebraic way via what we shall call the general multi-indices. For a set $I$, we define the set of multi-indices over $I$, denoted by $M(I)$, as the set of all finitely supported maps $w\colon I\to\mathbb{N}$, that is the maps for which all but finitely many $i\in I$ are mapped to $0$. We shall now consider $M(\mathbb{N}^{d+1})$, multi-indices over $\mathbb{N}^{d+1}$, which one can regard as commutative monomials in letters $\bn\in\mathbb{N}^{d+1}$.

Let us now introduce formal variables $z_{(\mfl,w)}$, $(\mfl,w) \in \mfL^-\times M(\mathbb{N}^{d+1})$, and define the \emph{general multi-indices} $\beta$ as monomials
\begin{equs} \label{general_multi_indices}
  z^{\beta} : = \prod_{(\mfl,w) \in \mfL^-\times M(\mathbb{N}^{d+1})}   z_{(\mfl,w)}^{\beta(\mfl,w)}. 
\end{equs}
These multi-indices have been introduced in the recent work \cite{BL23}. 
The vector space spanned by all monomials representing general multi-indices is just the polynomial algebra in variables $z_{(\mfl,w)}$, $(\mfl,w) \in \mfL^-\times M(\mathbb{N}^{d+1})$.  That algebra has, for each $\bn\in\mathbb{N}^{d+1}$, the derivation $D^{(\bn)}$ defined by the formula
 \begin{equs} \label{derivations}
D^{(\bn)}=\sum_{(\mfl,w) \in \mfL^-\times M(\mathbb{N}^{d+1})} (w({\bf{n}}) +1 ) \, z_{(\mfl,w+e_{\bn})}\partial_{z_{(\mfl,w)}}.
 \end{equs}
On the original level of polynomials in iterated partial derivatives of $a^\mfl (\mathbf{u})$, this derivation corresponds to applying one extra derivative $\partial_{u^{(\bn)}}$; as in the formula for the derivation $D$ in the introduction, the coefficient $w({\bf{n}})+1$ here corresponds to one of the possible choices of a basis.  

We shall be only interested in meaningful multi-indices called \emph{populated} satisfying the condition
\begin{equs}\label{populated_bis}
\sum_{(\mfl,w)} (1 - \length{w})\beta(\mfl,w)  = 1
\end{equs}
where  $\length{w}$ is the degree of $w$, viewed as a commutative monomial in letters $\bn\in\mathbb{N}^{d+1}$. The vector space spanned by all populated indices is denoted $\mathcal{M}^{\mathrm{gen}}$. We can define products $\triangleright_\bn$, $\bn\in\mathbb{N}^{d+1}$, on $\mathcal{M}^{\mathrm{gen}}$ by setting
\begin{equs} \label{def_products}
z^{\beta} \triangleright_\bn z^{\beta'} = z^{\beta} D^{(\bn)} (z^{\beta'}). 
\end{equs}
 Our goal now is to have a better understanding of this family of products.

\begin{proposition}\label{multi_Nov_example}
Let $A$ be a set and $R$ be a commutative associative algebra, and suppose that $\partial_a$, $a\in A$, are pairwise commuting derivations of that algebra. The bilinear operations $\triangleright_a$ on $R$ defined by the formulas
 \[
x\triangleright_a y:=x\partial_a(y)
 \]
satisfy the identities
\begin{gather*}
 (x \triangleright_a y) \triangleright_b z  - x \triangleright_a (y \triangleright_b z) = (y \triangleright_a x) \triangleright_b z  - y \triangleright_a (x \triangleright_b z), 
   \\ (x \triangleright_a y) \triangleright_b z  - x \triangleright_a (y \triangleright_b z) = (x\triangleright_b y) \triangleright_a z - x \triangleright_b (y \triangleright_a z),
\\   (x \triangleright_a y) \triangleright_b z   = (x \triangleright_b z)\triangleright_a y,
\end{gather*}
for all $ a,b\in A$.
\end{proposition}

\begin{proof}
Note that 
 \[
(x \triangleright_a y) \triangleright_b z=x\partial_a(y)\partial_b(z)=x\partial_b(z)\partial_a(y)=(x \triangleright_b z)\triangleright_a y,
 \]
so the last identity holds. Moreover, 
 \[
x \triangleright_a (y \triangleright_b z)=x\partial_a(y\partial_b(z))=x\partial_a(y)\partial_b(z)+xy\partial_a\partial_b(z),
 \]
so
 \[
(x \triangleright_a y) \triangleright_b z  - x \triangleright_a (y \triangleright_b z)= -xy\partial_a\partial_b(z),
 \]
which is symmetric in $x,y$ (since $R$ is commutative) and in $a,b$ (since the derivations $\partial_a$ commute between themselves), proving the first two identities. 
\end{proof}

If the set $A$ consists of one element $a$, these identities reduce to the identities 
\begin{gather*}
 (x \triangleright_a y) \triangleright_a z  - x \triangleright_a (y \triangleright_a z) = (y \triangleright_a x) \triangleright_a z  - y \triangleright_a (x \triangleright_a z), 
\\   (x \triangleright_a y) \triangleright_a z   = (x \triangleright_a z)\triangleright_a y,
\end{gather*}
which are the identities of the so-called (left) Novikov algebras \cite{BN85,DL,O92}. It is thus natural to give the following definition.

\begin{definition} \label{def_multi_novikov}
Let $A$ be a set. A \emph{(left) $A$-multi-Novikov algebra} is a vector space equipped with bilinear products $x,y\mapsto x\triangleright_a y$, $a\in A$, satisfying the identities
\begin{gather}
 (x \triangleright_a y) \triangleright_b z  - x \triangleright_a (y \triangleright_b z) = (y \triangleright_a x) \triangleright_b z  - y \triangleright_a (x \triangleright_b z)\label{eq:exchgen}, 
   \\ (x \triangleright_a y) \triangleright_b z  - x \triangleright_a (y \triangleright_b z) = (x\triangleright_b y) \triangleright_a z - x \triangleright_b (y \triangleright_a z)\label{eq:exchA},
\\   (x \triangleright_a y) \triangleright_b z   = (x \triangleright_b z)\triangleright_a y\label{eq:rcom},
\end{gather}
for all $ a,b\in A$.
\end{definition}

Note that the first two of the multi-Novikov identities imply the identity
 \[
(x \triangleright_a y) \triangleright_b z  - x \triangleright_a (y \triangleright_b z) = (y \triangleright_b x) \triangleright_a z  - y \triangleright_b (x \triangleright_a z),
 \]
which is the defining identity of left multi-pre-Lie algebras \cite[Proposition~4.21]{BCCH}. 

From Proposition~\ref{multi_Nov_example}, the vector space $\mathcal{M}^{\mathrm{gen}}$ equipped with the products $ (\triangleright_{{\bf{n}}})_{{\bf{n}} \in \mathbb{N}^{d+1}}  $ is a $A$-multi-Novikov algebra with $ A =  \mathbb{N}^{d+1}$.
Our main result on multi-indices is the following theorem.

\begin{theorem} \label{main_theorem_multi_indices}
The $A$-multi-Novikov algebra of populated general multi-indices is isomorphic to the free algebra generated by the set $\mfL^-$.
\end{theorem}

To prove this theorem, we shall undertake an intricate study of free multi-Novikov algebras. This is done in the following section.

\section{Free multi-Novikov algebras}
\label{sec::3}
The only source of multi-Novikov algebras that we have seen so far is given by Proposition \ref{multi_Nov_example}, which suggests to consider the following algebraic structure. 

\begin{definition}
Let $A$ be a set. A commutative multidifferential algebra is a vector space $R$ equipped with a commutative associative product and endomorphisms $\partial_a$, $a\in A$, which are pairwise commuting derivations of the product on $R$.
\end{definition}

Note that it is more common to use this structure to define a single binary operation on the different vector space $R\otimes\mathrm{Vect}(A)$, see, for instance \cite[Section~2.1]{Bur}; in our work, we use it to define many operations on the same vector space $R$.

Using this terminology, we may re-state Proposition \ref{multi_Nov_example} by saying that every commutative multidifferential algebra has a natural multi-Novikov structure, which we shall call the \emph{canonical multi-Novikov structure}. 

The following almost obvious proposition describes the free commutative multidifferential algebra generated by a set $X$, which we shall denote $\CDi_A\langle X\rangle$. For convenience of the reader, we recall that the notation $M(-)$ refers to multi-indices, that is finitely supported maps to $\mathbb{N}$.

\begin{proposition}
The algebra $\CDi_A\langle X\rangle$ is isomorphic to the polynomial algebra on the set of generators $X\times M(A)$ equipped with the derivations $\partial_a$ defined on generators as follows: for $(x,m)\in X\times M(A)$, we have
 \[
\partial_a(x,m)=(x,m_a),
 \]
where $m_a(a')=m(a')+\delta_{a,a'}$. In other words, the generator $(x,m)$ of the polynomial algebra is identified with the element $\prod_{a\in A}\partial_a^{m(a)}(x)$ of $\CDi_A\langle X\rangle$. 
\end{proposition}

The main result of this section is the following theorem that describes the free multi-Novikov algebra generated by a set $X$, which we shall denote $\Nov_A\langle X\rangle$. This theorem is a generalisation of \cite[Theorem~7.8]{DL}, and follows a similar strategy; however, given how intricate the proof is, we deemed important to give it with all the details.

\begin{theorem}\label{th:freeNovikov}
The algebra $\Nov_A\langle X\rangle$ is isomorphic to the multi-Novikov subalgebra of the free commutative multidifferential algebra $\CDi_A\langle X\rangle$ generated by $X$; here we consider $\CDi_A\langle X\rangle$ with its canonical multi-Novikov structure.
\end{theorem}

The proof of this theorem consists of several parts. First, let us describe the multi-Novikov subalgebra of the free commutative multidifferential algebra $\CDi_A\langle X\rangle$ generated by $X$ in a more concrete way. 

\begin{proposition}\label{prop:lowerbound}
The multi-Novikov subalgebra of the free commutative multidifferential algebra $\CDi_A\langle X\rangle$ generated by $X$ is spanned by all monomials 
$\prod_{x,m}(x,m)^{n_{x,m}}$ for which the ``populated condition'' holds:
 \[
\sum_{x,m} n_{x,m}\left(-1+\sum_{a\in A}m(a)\right)=-1.
 \]
\end{proposition}

\begin{proof}
We shall freely use the identification of generators $x\in X$ with $(x,0)$ in $\CDi_A\langle X\rangle$. Note that the populated condition holds for each generator $x\in X$: we have $m(a)=0$ for all $a\in A$ and $n_{x,m}=1$. Moreover, if we have two monomials $u,v$ for which the populated condition holds, then $u\triangleright_a v=u\partial_a(v)$ is a sum of monomials for which the populated condition holds as well, so the vector space spanned by such monomials is a multi-Novikov subalgebra. Let us show that this subalgebra is generated by~$X$. Let us take some monomial $u$ satisfying the populated condition. If for some $a\in A$ and some $x\in X$ we have $u=\partial_a(x)u_1$, we have $u=u_1\triangleright_a x$, and we may argue by induction on degree. Otherwise, $u$ is a product of a monomial in variables $X$ and a monomial in variables $(x,m)$ where $\sum_{a\in A}m(a)\ge 2$. Let us take one of the latter, say $(x_0,m_0)$ with $\sum_{a\in A}m_0(a)=d\ge 2$, and take some $a_0$ for which $m_0(a_0)\ne 0$. Because of the populated condition, the monomial $u$ is also divisible by a monomial of degree at least $d$ in variables from~$X$, say $u=x_{i_1}\cdots x_{i_d}(f_0,x_0)u_1$. Note that the populated condition is satisfied for~$u_1$.
Let us define $m_1\in M(A)$ by the formula $m_1(a)=m_0(a)-\delta_{a_0,a}$. Note that the populated condition is satisfied for $x_{i_1}\cdots x_{i_d}(x,m_1)$, and we have
\begin{multline}
u_1\triangleright_{a_0}(x_{i_1}\cdots x_{i_d}(x,m_1))=\partial_{a_0}(x_{i_1}\cdots x_{i_d}(x,m_1))u_1=\\ u+\sum_{j=1}^d (x_{i_1}\cdots x_{i_{j-1}} x_{i_{j+1}}\cdots x_{i_d}(x,m_1)u_1)\triangleright_{a_0}x_{i_j}
\end{multline}
and we once again may argue by induction on degree.
\end{proof}

 We shall use Proposition \ref{prop:lowerbound} to give a lower bound on the ``size'' of the free multi-Novikov algebra. To make sense of it, we define a $\mathbb{Z}\times \mathbb{N}$-grading on $\CDi_A\langle X\rangle$ by postulating that each variable $(x,m)$ has the grading 
 \[
\left(-1+\sum_{a\in A}m(a),1\right),
 \]   
and extending this grading to products of variables additively. If we assume the sets $A$ and $X$ finite, the bi-homogeneous components of $\CDi_A\langle X\rangle$ are finite-dimensional\footnote{We shall be making this assumption from now on; for the interested reader, we indicate that in the case where (say) $X$ is infinite, one may consider a finer grading by $\mathbb{N}^X$ instead of $\mathbb{N}$, which allows for all our proofs to proceed without any problem.}, and Proposition \ref{prop:lowerbound} immediately implies the following lower bound.

\begin{corollary} \label{lower_bound}
For each $n\in\mathbb{N}$ the dimension of the degree $n$ component of $\Nov_A\langle X\rangle$ is greater than or equal to the dimension of the $(-1,n)$-homogeneous component of $\CDi_A\langle X\rangle$.
\end{corollary}

We shall now obtain the same upper bound for the dimension of the degree $n$ component of $\Nov_A\langle X\rangle$, which we shall use to prove Theorem \ref{th:freeNovikov}. We shall for the moment perform all calculations in the multi-magmatic algebra $\Mag_A\langle X\rangle$ which has bilinear operations $\triangleright_a$, $a\in A$, without any identities between them. This algebra has an obvious basis of monomials which are binary trees whose internal vertices are labelled by elements of $A$ and whose leaves are labelled by elements of $X$.

The algebra $\Mag_A\langle X\rangle$  has an ideal $I$ generated by all elements
\begin{gather}
 (x \triangleright_a y) \triangleright_b z  - x \triangleright_a (y \triangleright_b z) - (y \triangleright_a x) \triangleright_b z  + y \triangleright_a (x \triangleright_b z), 
   \\ (x \triangleright_a y) \triangleright_b z  - x \triangleright_a (y \triangleright_b z) - (x\triangleright_b y) \triangleright_a z + x \triangleright_b (y \triangleright_a z),
\\   (x \triangleright_a y) \triangleright_b z   - (x \triangleright_b z)\triangleright_a y,
\end{gather}
for all $ a,b\in A$ and all $x,y,z\in \Mag_A\langle X\rangle$. The quotient by this ideal is clearly the free multi-Novikov algebra $\Nov_A\langle X\rangle$, and we shall now obtain a set of elements that spans the quotient $\Mag_A\langle X\rangle/I$. 

Let $y_1,\ldots,y_{n+1}\in \Mag_A\langle X\rangle$ be some monomials. Recall that monomials $y_1\triangleright_{a_1} (y_2\triangleright_{a_2} (\cdots (y_n\triangleright_{a_n} y_{n+1}) \cdots ))$ are often referred to as right-normed products of $y_1,\ldots,y_{n+1}$. Each monomial in $\Mag_A\langle X\rangle/I$ can be uniquely written as a right-normed product $y_1\triangleright_{a_1} (y_2\triangleright_{a_2} (\cdots (y_n\triangleright_{a_n} z) \cdots ))$, where $z\in X$; for typographic purposes, we shall denote such product as $r(y_1,a_1,\ldots,y_n,a_n;z)$. 

\begin{definition} \label{leaning}
A right-normed product $r(y_1,a_1,\ldots,y_n,a_n;z)$ with $z\in X$ is said to be \emph{left-leaning} if $y_2,\ldots,y_n\in X$, and $y_{1}$ is either a generator or a monomial whose right-normed product decomposition is itself left-leaning. 
\end{definition}

Let us furthermore specify a subclass of left-leaning products. 

\begin{definition} \label{ordered}
Let us fix an order of the set $A$ and an order of the set $X$. We shall say that a left-leaning product $r(y_1,a_1,\ldots,y_n,a_n;z)$ is \emph{ordered} if $n=0$ (so that the whole product is just the generator $z$) or $n>0$, $a_1\le a_2\le\ldots\le a_n$, and $y_1\in X$ or $y_1=r(y_1',a_1',\ldots,y_m',a_m';z')$ is an ordered left-leaning product satisfying one of the two conditions:
\begin{itemize}
\item $n<m$,
\item $n=m$ and $z<z'$. 
\end{itemize} 
\end{definition}

\begin{proposition}\label{prop:ordright}
The quotient $\Mag_A\langle X\rangle/I$ is spanned by cosets of ordered left-leaning products. 
\end{proposition}

\begin{proof}
Let us prove that in $\Mag_A\langle X\rangle/I$ every coset of a monomial $u$ of degree $d$ can be written as a linear combination of cosets of monomials of degree $d$ which are ordered left-leaning products. We shall argue by induction on $d$, and for any fixed $d$ by induction on $n$, the length of the right-normed product decomposition $r(y_1,a_1,\ldots,y_n,a_n;z)$ of our monomial~$u$. 

Clearly, Identity \eqref{eq:exchA} implies that we can put $a_1,\ldots,a_n$ in the right order at the cost of elements with same $d$ and smaller $n$ that appear (to which the induction hypothesis applies). Similarly, Identity \eqref{eq:exchgen} implies that we can put $y_1,\ldots,y_{n}$ in the decreasing order according to their degrees at the cost of elements with same $d$ and smaller $n$ that appear (to which the induction hypothesis applies). If the resulting element 
 \[
u'=r(z_1,b_1,\ldots,z_n,b_n;z)
 \]
satisfies $z_1\in X$, then the fact that the degrees of $z_1,\ldots,z_{n}$ are in the decreasing order implies that we have $z_i\in X$ for all $i$, and so $u'$ is an ordered right-leaning product. If $z_1\notin X$, then the induction hypothesis applies to it, and its coset contains a linear combination of ordered right-leaning products. Without loss of generality, we may take
 \[
z_1=r(z_1',b_1',\ldots,z_m',b_m';z')
 \]
with $z_2'$,\ldots, $z_m'\in X$. Using Identity \eqref{eq:rcom} for $a=b_1$ and $b=b_1'$, we see that in the quotient $\Mag_A\langle X\rangle/I$ we have  
 \[
u'=r(r(z_1',b_1,z_2,\ldots,z_n,b_n;z),b_1',z_2',\ldots,z_m',b_m';z')
 \]
Now we can use the induction hypothesis to write the monomial
 \[
r(z_1',b_1,z_2,\ldots,z_n,b_n;z) 
 \]
as a linear combination of ordered right-leaning products. To ensure the same for the monomial $u'$, we may have to use Identity \eqref{eq:rcom} once again, completing the proof.  
\end{proof}

We are now ready to describe what will be a basis of the free multi-Novikov algebra. For that, we shall first recall and use some standard combinatorial argument (Knuth's rotation correspondence \cite[Chapter~2.3]{Kn}) relating planar rooted binary trees and planar rooted trees. Let $T$ be a planar rooted binary tree with $n$ leaves. We define a planar rooted tree $\rho(T)$ with $n$ vertices as follows: for each internal vertex $i$ of $T$, consider the path going from $i$ to the rightmost leaf of the tree growing from $i$, and collapse that path to a single vertex. An example of this is displayed in Figure \ref{fig1}, where we use the same label for all the vertices of the planar binary tree collapsed to the same one by the map $\rho$.
\begin{figure}[h]
  \centering
\begin{tikzcd}[sep=small]
  &&& 6 && 2 \\
  5 && 4 && 2 \\
  & 4 &&&&&&&& 5 \\
  1 && 2 &&&&&& 1 & 4 & 6 \\
  & 2 && 3 &&&&&& 2 \\
  && 3 &&&&&&&& 3
  \arrow[no head, from=1-4, to=2-5]
  \arrow[no head, from=2-1, to=3-2]
  \arrow[no head, from=2-5, to=1-6]
  \arrow[no head, from=3-2, to=2-3]
  \arrow[no head, from=3-2, to=4-3]
  \arrow[no head, from=3-10, to=4-10]
  \arrow[no head, from=4-3, to=2-5]
  \arrow[no head, from=4-10, to=5-10]
  \arrow[no head, from=4-11, to=5-10]
  \arrow[no head, from=5-2, to=4-1]
  \arrow[no head, from=5-2, to=4-3]
  \arrow[no head, from=5-10, to=4-9]
  \arrow[no head, from=6-3, to=5-2]
  \arrow[no head, from=6-3, to=5-4]
  \arrow[no head, from=6-11, to=5-10]
\end{tikzcd}
\caption{Knuth's rotation correspondence}\label{fig1}
\end{figure} 

If $T$ is the underlying planar binary tree of a monomial $u\in \Mag_A\langle X\rangle$, we have a labelling of all internal vertices of $T$ by elements of the set $A$ and a labelling of all leaves of $T$ by elements of the set $X$. These labellings induce the labelling of vertices and edges of the planar tree $\rho(T)$: each vertex of $\rho(T)$ acquires the label of the leaf of $T$ which we collapsed onto that vertex, and each edge of $\rho(T)$ acquires the label of the vertex of $T$ for which that edge is the left incoming edge. An example of this is displayed in Figure \ref{fig2}.

\begin{figure}[h]
  \centering
\begin{tikzcd}[sep=small]
  &&& x && x \\
  y && x && a \\
  & b &&&&&&&& y \\
  x && b &&&&&& x & x & x \\
  & a && y &&&&&& x \\
  && b &&&&&&&& y
  \arrow[no head, from=1-4, to=2-5]
  \arrow[no head, from=2-1, to=3-2]
  \arrow[no head, from=2-5, to=1-6]
  \arrow[no head, from=3-2, to=2-3]
  \arrow[no head, from=3-2, to=4-3]
  \arrow["b"', no head, from=3-10, to=4-10]
  \arrow[no head, from=4-3, to=2-5]
  \arrow["b"', no head, from=4-10, to=5-10]
  \arrow["a"', no head, from=4-11, to=5-10]
  \arrow[no head, from=5-2, to=4-1]
  \arrow[no head, from=5-2, to=4-3]
  \arrow["a", no head, from=5-10, to=4-9]
  \arrow[no head, from=6-3, to=5-2]
  \arrow[no head, from=6-3, to=5-4]
  \arrow["b", no head, from=6-11, to=5-10]
\end{tikzcd}
\caption{Knuth's rotation correspondence with labels}\label{fig2}
\end{figure} 

We shall refer to an occurrence of a generator $x\in X$ in $u$ as a \emph{root} if it becomes a label of an internal vertex of $\rho(T)$, and a \emph{leaf} it it becomes a label of a leaf of $\rho(T)$. 

\begin{definition} \label{normal}
We shall call an ordered left-leaning product \emph{normal} if the sequence of its leaves, read from the left to the right, is in the increasing order with respect to the order of $X$.
\end{definition}

We are now ready to prove the main technical result of this section.

\begin{proposition} \label{normal_coset}
The quotient $\Mag_A\langle X\rangle/I$ is spanned by cosets of normal ordered left-leaning products. 
\end{proposition}

\begin{proof}
Thanks to the result of Proposition \ref{prop:ordright}, it is enough to prove that every ordered left-leaning product can be written as
a linear combination of cosets of normal ordered left-leaning products. 
Let us take an ordered left-leaning product 
 \[
u=r(y_1,a_1,\ldots,y_n,a_n;z) 
 \]
with $a_1\le a_2\le\ldots\le a_n$, $y_2,\ldots,y_n\in X$, and either $y_{1}\in X$ or $y_{1}=
r(y_1',a_1',\ldots,y_m',a_m';z')$ satisfies $n<m$ or $n=m$ and $z<z'$. We shall prove that in $\Mag_A\langle X\rangle/I$ it can be written as a linear combination of cosets of normal ordered left-leaning products by induction on $d$, and for a fixed $d$ by induction on $\ell$, the number of leaves, and for fixed $d$ and $\ell$ by the sum $n+m$.

Suppose first that $y_{1}\in X$, so that all the elements $y_1,\ldots,y_{n}$ are leaves. Using Identity \eqref{eq:exchgen}, we may put them in order at the cost of elements with fewer leaves that appear. We may apply the procedure of Proposition \ref{prop:ordright} (which does not increase the number of leaves) to these elements, and then the induction hypothesis finishes the argument. 

Suppose now that $y_1\notin X$, so that we have $y_{1}=r(y_1',a_1',\ldots,y_m',a_m';z')$ with $n<m$ or $n=m$ and $z<z'$. By the induction hypothesis we may assume that $y_{1}$ is a normal ordered left-leaning product. Also, modulo terms with fewer leaves, we may assume $y_2\le y_3\le\ldots\le y_n$. Suppose $y_m'> y_2$. If we show that these two elements can be exchanged modulo $I$, the proof will be finished after doing several exchanges like that. 

Using our relations, we can change the order from $r(y_1',a_1',\ldots,y_m',a_m';z')$ to $r(y_m',a_1',y_1'\ldots,y_{m-1}',a_m';z')$
at the cost of some elements with fewer leaves, and one more element $r((y_1'\triangleright a_1' y_m'),a_2,\ldots,y_{m-1}',a_m';z')$ with the same number of leaves. However, this element has fewer factors in the second level right-normed factorisation, and the induction on $n+m$ takes care of that. Thus, we may focus on the element
 \[
r(r(y_m',a_1',y_1'\ldots,y_{m-1}',a_m';z'),a_1,y_2,\ldots,y_n,a_n;z),
 \]
which, thanks to Identity \eqref{eq:rcom}, is equal to
 \[
r(r(y_m',a_1,y_2,\ldots,y_n,a_n;z),a_1',y_1'\ldots,y_{m-1}',a_m';z'),
 \]
Now, as before, we use Identity \eqref{eq:exchgen} to exchange in $r(y_m',a_1,y_2,\ldots,y_n,a_n;z)$ the elements $y_m'$ and $y_2$, and use Identity \eqref{eq:rcom} again to obtain
 \[
r(r(y_2,a_1',y_1'\ldots,y_{m-1}',a_m';z'),a_1,y_m',\ldots,y_n,a_n;z),
 \]
and an argument similar to before allows us to exchange $y_2$ and $y_{1}'$ in the inner monomial $r(y_2,a_1',y_1'\ldots,y_{m-1}',a_m';z')$, leading to a linear combination of normal ordered left-leaning products by the induction hypothesis.
\end{proof}

\begin{proof}[of Theorem \ref{th:freeNovikov}]
Let us show that normal ordered left-leaning products of degree $d$ are in bijection with the standard basis elements of the $(-1,d)$-homogeneous component of $\CDi_A\langle X\rangle$. The way to construct a basis element out of an ordered left-leaning product is quite straightforward: to each root vertex labelled by $x$ whose operations in the right-normed factorisation are $\triangleright_{a_1}$,\ldots, $\triangleright_{a_n}$ with $a_1\le a_2\le a_n$, we associate the variable $\partial_{a_1}\cdots\partial_{a_n}(x)$, and to each leaf labelled by $x$ we associate simply the variable $x$. The reverse reconstruction is straightforward (the ordered condition of the left-leaning product ensures uniqueness), and leads to a valid monomial in $\Mag_A\langle X\rangle$ thanks to the populated condition. 

Since dimensions of bi-homogeneous components of $\CDi_A\langle X\rangle$ are finite-dimensional, we see that the surjective map 
 \[
\Nov_A\langle X\rangle\twoheadrightarrow\CDi_A\langle X\rangle_{-1}
 \]
must be an isomorphism, as required.
\end{proof}

\begin{proof}[of Theorem \ref{main_theorem_multi_indices}]
\label{proof_multi}
We just have to realise that populated general multi-indices are exactly the  multi-Novikov subalgebra of the free commutative multidifferential algebra described in Proposition~\ref{prop:lowerbound}. We provide a dictionary between the two objects:
\begin{equs}
	X = \mfL^-, \quad A =\mathbb{N}^{d+1} \quad m = w, \quad x = \mfl, \quad (x,m) = z_{(\mathfrak{l},w)},\quad n_{x,m} = \beta(\mathfrak{l},w),
\end{equs}
where for $ m=w $, we identify the map $m$ that counts the occurence of the elements $ a $ with the word $ w $ in the commutative letters $a$. One has
\begin{equs}
	\sum_{a \in A} m(a) = \length{w}
\end{equs}
Then, one also has
\begin{equs}
	\prod_{x,m}(x,m)^{n_{x,m}} =  z^{\beta}  = \prod_{(\mfl,w) \in \mfL^-\times M(\mathbb{N}^{d+1})}   z_{(\mfl,w)}^{\beta(\mfl,w)}. 
\end{equs} 
and one has the same populated condition 
\[
\sum_{x,m} n_{x,m}\left(-1+\sum_{a\in A}m(a)\right) = \sum_{(\mfl,w)} (-1 + \length{w})\beta(\mfl,w) =-1.
\]
which allows to conclude the proof.
\end{proof}

\section{SPDE multi-indices}

\label{sec::4}

We shall now extend the context of Section \ref{sec::2} to include, in addition to the derivatives $\partial_{u^{(\bn)}}$, also the derivatives $\partial_{x_i}$ which are more complicated and computed via the chain rule
 \[
\partial_{x_i}=\sum_{\bn\in\mathbb{N}^{d+1}}(n_i+1)u^{(\bn+e_i)}\partial_{u^{(\bn)}}.
 \]

Together with the derivatives $\partial_{u^{(\bn)}}$, these satisfy the following relations:
\begin{equs} \label{relation_derivs}
	\begin{aligned}
\partial_{x_i}\partial_{x_j} &=\partial_{x_j}\partial_{x_i},\quad \partial_{u^{(\bn)}}\partial_{u^{(
		\mathbf{m})}}=\partial_{u^{(\mathbf{m})}}\partial_{u^{(\bn)}}, \\  \partial_{u^{(\bn)}}\partial_{x_i} & =\partial_{x_i}\partial_{u^{(\bn)}} + n_i \partial_{u^{(\bn-e_i)}},
	\end{aligned}
\end{equs}
where $e_i$ is the standard basis vector of $\mathbb{N}^{d+1}$ and $ \mathbf{n}, \mathbf{m} \in \mathbb{N}^{d+1} $.

Previously, we used multi-indices over $\mathbb{N}^{d+1}$, which we regarded as monomials in letters $\bn\in\mathbb{N}^{d+1}$. Now, we have monomials that use those letters but additionally letters $d_i$, $0\le i\le d$, but not all of these letters commute with each other. For that reason, we introduce an abstract associative algebra $\cA$ generated by all these symbols, and impose the relations 
\begin{equs} \label{relation_words}
d_id_j=d_jd_i,\quad \bn\mathbf{m}=\mathbf{m}\bn,\quad  d_i \bn = n_i(\bn-e_i) + \bn d_i.
\end{equs}
For the convenience of the reader, we emphasize that in the last equation $n_i$ is the $i$-th component of $\bn$, so that the corresponding term is missing if $n_i=0$, and for $n_i>0$, we have $\bn-e_i\in\mathbb{N}^{d+1}$, and we consider $(\bn-e_i)$ as a letter in the alphabet~$\mathbb{N}^{d+1}$. 
Since the last group of relations in this algebra replaces a monomial by a linear combination of two monomials, we are led to introduce a completely new type of multi-indices, which we shall call SPDE multi-indices. We shall use a set of formal variables $\coord=(z_{(\mfl,\alpha)})_{(\mfl,\alpha) \in \mfL^-\times\cA}$, which we consider to be linear in the argument $\alpha$, so that 
  \[
z_{(\mfl,c_1\alpha_1+c_2\alpha_2)}=c_1z_{(\mfl,\alpha_1)}+c_2z_{(\mfl,\alpha_2)}.
 \]
Each such variable $ z_{(\mfl,\alpha)} $ corresponds to $\mathrm{D}^{\alpha} a^\mfl(\mathbf{u}(x))$, where $\mathrm{D}^\alpha$ is obtained from $\alpha$ by replacing $d_i$ with $\partial_{x_i}$ and $\bn$ with $\partial_{u^{(\bn)}}$. Multi-indices $\beta$ over $\coord$  measure the frequency of the variables $  z_{(\mathfrak{l},\alpha)} $, so that we can represent them by monomials
\begin{equs}
	z^{\beta} : = \prod_{(\mfl,\alpha) \in \coord}  z_{(\mfl,\alpha)}^{\beta(\mfl,\alpha)}. 
\end{equs}
Note that Relations \eqref{relation_words} imply linear dependencies between the generators $z_{(\mfl,\alpha)}$ and hence between these monomials, so the right object to consider, which we shall refer to as the \emph{SPDE multi-indices} is the vector space spanned by all such monomials. Moreover, we are interested only in meaningful multi-indices called \emph{populated} satisfying the condition
\begin{equs}\label{populated}
	\sum_{(\mfl,\alpha)} (1 - \length{\alpha})\beta(\mfl,\alpha)  = 1.
\end{equs}
where  $\length{\alpha}$ is the number of letters $\bn\in\mathbb{N}^{d+1}$ in	$\alpha$. Note that while Relations \eqref{relation_words} imply some dependencies betweel the variables $z_{(\mfl,\alpha)}$, those relations are homogeneous with respect to the degree $\length{\alpha}$, so the populated condition is well defined, and we can define the vector space $\mathcal{M}_{\mathcal{R}}$ as the span of all populated SPDE multi-indices. 

\begin{remark} \label{remark_z_k}
Usually, the approach using multi-indices \cite{OSSW,LOT,BL23} suggests to encode iterated partial derivatives $\partial_{x_i}$ by monomials in another set of variables $z_{\bn}$, $\bn\in\mathbb{N}^{d+1}$, which correspond to the $u^{(\bn)}$. Our coding is different: we do not apply the chain rule, but rather look at iterated partial derivatives of the non-linearities $a^\mfl(\mathbf{u})$, which leads to  more compact formulas. For example, our multi-index $ z_{(\mathfrak{l},d_i)} $ corresponds to
	\begin{equs}
		\partial_i a^\mfl(\mathbf{u}(x)) = \sum_{\bn} u^{(\bn+e_i)}  \partial_{u^{(\bn)}}  a^\mfl(\mathbf{u}(x))
	\end{equs}
which would otherwise correspond to a more complicated multi-index: the sum of terms $ z_{\bn+e_i} z_{(\mathfrak{l},\bn)} $.
\end{remark}

Let us introduce a family of derivations on the space of SPDE multi-indices; we shall denote them $D^{(\bn)}$, $\bn \in \mathbb{N}^{d+1}$, and $\partial_i$, $0\le i\le d$. The derivation $D^{(\bn)}$ is the unique derivation that sends every variable $z_{(\mathfrak{l},\alpha)}$ to $z_{(\mathfrak{l},\bn\alpha)}$, and the derivation $\partial_{i}$ is the unique derivation that sends every variable $z_{(\mathfrak{l},\alpha)}$ to $z_{(\mathfrak{l},d_i\alpha)}$. Note that these derivations of course respect the linear relations between the variables $z_{(\mathfrak{l},\alpha)}$ coming from Relations \eqref{relation_words}. Moreover, these relations immediately imply that
\begin{equs} \label{non-commutation_2}
  D^{(\bn)} \partial_i  =\partial_i D^{(\bn)} + n_i D^{(\bn-e_i)}.
\end{equs}

Using the commuting derivations $D^{(\bn)}$, we can still define a family of products $\triangleright_\bn$ on the vector space of all SPDE  multi-indices by setting
\begin{equs}
z^{\gamma} \triangleright_\bn z^{\gamma'} = z^{\gamma} D^{(\bn)} (z^{\gamma'}). 
\end{equs}
It is immediate that each such product preserves the populated condition, and hence defines a multi-Novikov algebra structure on $\mathcal{M}_{\mathcal{R}}$. 

Our second main result on multi-indices is the following theorem for $A = \mathbb{N}^{d+1}$.

\begin{theorem} \label{second_main_theorem_multi_indices}
The $A$-multi-Novikov algebra of populated SPDE multi-indices is isomorphic to the free algebra generated by the set $\mathbb{N}^{d+1}\times \mfL^-$.
\end{theorem}

This theorem will rely on a general result on extended multi-Novikov algebras which will be proved in the next section.

\section{Extended free multi-Novikov algebras}
\label{sec::5}
Let us consider some type of algebras $\cP_A$ with operations indexed by a set $A$. We may consider a linearised type of algebras $\cP_A^{\mathrm{lin}}$, where instead of operations $f_a$ indexed by $a\in A$ we consider operations $f_v$ indexed by $v\in\mathrm{Vect}(A)$, where we postulate that the operations $f_v$ satisfy the exact same identities as the operations $f_a$ and additionally require $f_{\lambda_1a_1+\lambda_2a_2}=\lambda_1f_{a_1}+\lambda_2f_{a_2}$. One of the most famous examples of this sort concerns linearly compatible Lie brackets. Recall that an algebra with linearly compatible Lie brackets \cite{DK} is a vector space with anti-commutative operations $x,y\mapsto \{x,y\}_a$, $a\in A$, for which each linear combination $\lambda\{x,y\}_{a}+\mu\{x,y\}_{b}$ is a Lie bracket. 

\begin{proposition}
The class of algebras with several compatible Lie brackets can be alternatively described as the class of algebras with anti-commutative operations $x,y\mapsto \{x,y\}_a$, $a\in A$, such that each operation $\{x,y\}_{a}$ is a Lie bracket and those Lie brackets satisfy, for all $a\ne b\in A$, the following generalization of the Jacobi identity:
\begin{multline*}
\{\{x,y\}_a,z\}_b+\{\{y,z\}_a,x\}_b+\{\{z,x\}_a,y\}_b\\ + \{\{x,y\}_b,z\}_a+\{\{y,z\}_b,x\}_a+\{\{z,x\}_b,y\}_a=0.  
\end{multline*}
\end{proposition}

\begin{proof}
The Jacobi identity for the operation $\lambda\{x,y\}_{a}+\mu\{x,y\}_{b}$ is a polynomial of degree two in $\lambda,\mu$. The coefficient of $\lambda^2$ is the Jacobi identity for $\{x,y\}_{a}$, the coefficient of $\mu^2$ is the Jacobi identity for $\{x,y\}_{b}$, and the coefficient of $\lambda\mu$ is precisely the mixed identity above. 
\end{proof}

Thus, if $\cP_A$ stands for vector spaces equipped with several Lie brackets indexed by $A$, $\cP_A^{\mathrm{lin}}$ imposes the above mixed identities between different brackets. For various multi-algebras, the situation is different.

\begin{proposition} \label{linearised_Novikov}
For the class of multi-pre-Lie algebras, or multi-Novikov algebras, and multidifferential algebras the linearised version is the same class of algebras.  
\end{proposition}

\begin{proof}
Let us consider, for example, the defining identity
 \[
 R(x,y,z,a,b):=(x \triangleright_a y) \triangleright_b z  - x \triangleright_a (y \triangleright_b z) - (y \triangleright_a x) \triangleright_b z  + y \triangleright_a (x \triangleright_b z)=0
 \]
for multi-pre-Lie algebras. For all $a_1,a_2,b_1,b_2\in A$ and all $\lambda_1,\lambda_2,\mu_1,\mu_2$, we have
\begin{multline}
R(x,y,z,\lambda_1a_1+\lambda_2a_2,\mu_1b_1+\mu_2b_2)=\\
\lambda_1\mu_1R(x,y,z,a_1,b_1)+\lambda_1\mu_2R(x,y,z,a_1,b_2)+\\
\lambda_2\mu_1R(x,y,z,a_2,b_1)+\lambda_2\mu_2R(x,y,z,a_2,b_2)=0,
\end{multline}
so no new identities arise.
\end{proof}

Suppose now that the vector space $V=\mathrm{Vect}(A)$ carries a representation of a Lie algebra $\mathfrak{g}$. 

\begin{definition} \label{def_action_g}
The class of \emph{$\mathfrak{g}$-extended $\cP_A^{\mathrm{lin}}$-algebras} has the same operations as $\cP_A^{\mathrm{lin}}$-algebras and, additionally, unary operations $\alpha_g$, $g\in\mathfrak{g}$ that satisfy the commutation relations $\alpha_g\alpha_h-\alpha_h\alpha_g=\alpha_{[g,h]}$ of the Lie algebra $\mathfrak{g}$ and the identities 
 \[
\alpha_g f_v(x_1,\ldots,x_n)=\sum_{i=1}^n f_v(x_1,\ldots,x_{i-1},\alpha_g(x_i),x_{i+1},\ldots,x_n)+f_{g(v)}(x_1,\ldots,x_n)
 \]
for each structure operation $f_v$, $v\in V$, with $n$ arguments. In other words, elements of $\mathfrak{g}$ act like generalised derivations, acting both on arguments and on operations. 
\end{definition}

This notion corresponds to semi-direct product extensions in the operad theory that goes back to \cite{Markl} in the topological context and to \cite{Bel} in the linear context.

\begin{proposition}\label{prop:extlin}
As a $\cP_A^{\mathrm{lin}}$-algebra, the free $\mathfrak{g}$-extended $\cP_A^{\mathrm{lin}}$-algebra generated by a vector space $W$ is isomorphic to the free algebra generated by $U(\mathfrak{g})\otimes W$, the free $\mathfrak{g}$-module on $W$.  
\end{proposition}

\begin{proof}
Clearly, the identities
\begin{equation}\label{eq:der-distr}
\alpha_g f_v(x_1,\ldots,x_n)=\sum_{i=1}^n f_v(x_1,\ldots,x_{i-1},\alpha_g(x_i),x_{i+1},\ldots,x_n)+f_{g(v)}(x_1,\ldots,x_n)
\end{equation}
imply, by induction on the number of generators from $W$ involved, that every element of the free $\mathfrak{g}$-extended $\cP_A^{\mathrm{lin}}$-algebra generated by $W$ can be written as a combination of operations of elements of $\cP_A^{\mathrm{lin}}$ evaluated on elements from $U(\mathfrak{g})\otimes W$, and we should ensure some compatibilities between our relations to ensure that these elements are linearly independent.

There are two things to check. First, we should check that Identities \eqref{eq:der-distr} are compatible with $\alpha_g\alpha_h-\alpha_h\alpha_g=\alpha_{[g,h]}$. That is clear: computing $\alpha_g\alpha_h f_v(x_1,\ldots,x_n)$ gives us terms where $g$ and $h$ are both applied in the same place (one of the arguments $x_i$ or $v$) or at different places, and subtracting $\alpha_h\alpha_g f_v(x_1,\ldots,x_n)$ has the effect that all terms where they are applied at different places disappear.
Second, we should check that Identities \eqref{eq:der-distr} are compatible with the identities of $\cP_A^{\mathrm{lin}}$-algebras. That is also clear, since applying $\alpha_g$ to such an identity produces a sum of terms which assemble into the same identity where $g$ is applied to either one operation or one of the arguments (as $\alpha_g$).
\end{proof}

\begin{proof}[of Theorem \ref{second_main_theorem_multi_indices}]
Let us consider $A=\mathbb{N}^{d+1}$, the vector space $V=\mathrm{Vect}(A)$ is identified with the polynomial algebra in $d+1$ variables which has derivations $\partial_0$, \ldots, $\partial_d$, making $V$ a module over the $(d+1)$-dimensional abelian Lie algebra $\mathfrak{g}$. In particular, Relations \eqref{non-commutation_2} are a particular case of Relation \eqref{eq:der-distr} for $n=1$ and $f_v=D^{(\bn)}$. Thus, analogously to what we saw in the proof of Theorem \ref{main_theorem_multi_indices}, the free $\mathfrak{g}$-extended $\CDi_A$-algebra generated by the set $\mfL^-$ is precisely the algebra of all SPDE multi-indices equipped with the commutative product and the operators $D^{(\bn)}$, $\bn\in \mathbb{N}^{d+1}$ and $\partial_i$, $0\le i\le d$.  

According to Proposition \ref{prop:extlin}, as a $\CDi_A$-algebra, the free $\mathfrak{g}$-extended $\CDi_A$-algebra generated by a vector space $W$ is isomorphic to the free algebra generated by $U(\mathfrak{g})\otimes W$. Since both the canonical multi-Novikov structure and the construction of extended algebras are functorial, Theorem \ref{th:freeNovikov} implies that, as a Novikov algebra, the free extended multi-Novikov algebra generated by a set $X$ is isomorphic to the free algebra generated by the vector space $U(\mathfrak{g})\otimes W$, where $W=\mathrm{Vect}(X)$. It remains to note that since $\mathfrak{g}$ is abelian, the vector space $U(\mathfrak{g})\otimes W$ has a combinatorial basis $\mathbb{N}^{d+1}\times X$, and setting $X=\mfL^-$, we obtain the required statement. 
\end{proof}

For the reader interested in applications, we extract from the above proof a more concrete version of the statement of Theorem \ref{second_main_theorem_multi_indices}: the generating set $\mathbb{N}^{d+1}\times \mfL^-$ from the statement of this theorem 
is obtained from the set $\mfL^-$ (corresponding to the non-linearities $a^\mfl(\mathbf{u})$) by applying monomials in partial derivatives $\partial_{x_0},\ldots,\partial_{x_{d}}$.

\section{Connection to decorated trees}
\label{sec::6}

Expansion of solutions of \eqref{set01} were first described with a B-series formalism in \cite{BCCH} using the decorated trees formalism introduced in \cite{reg,BHZ}. Decorated trees are used for encoding two different analytical/stochastic objects: elementary differentials, Taylor coefficients of the local expansion of the solution and recentered stochastic iterated integrals. In this section, we introduce a suitable set of decorated trees convenient for this purpose.
We start by considering planar decorated trees from 
\cite[Section 4]{BK23}. They are recursively defined by
\begin{equs}
	\mathfrak{T} = \Big\{  (\prod_{i} A_i) \Xi_{\mathfrak{l}}, \, \mfl \in \mfL^-, \, A_i \in  \{ \CI_a(\tau), \, \tau \in \mathfrak{T}, \, a\in \mathbb{N}^{d+1} \}  \cup \{ X_{i} \}_{i = 0,..., d} \Big\}.
\end{equs}
where 
\begin{itemize}
	\item $ \Xi_{\mathfrak{l}} $ is a noise type edge.
	\item $ \CI_a(\tau) $ corresponds to the grafting of the decorated tree $ \tau $ onto a new root with an edge decorated by $ a \in \mathbb{N}^{d+1} $. The root has no decorations.
	\item the $ X_i $ are monomial type edges.
	\item the product $ \Pi_i $ is not commutative and therefore $ \mathfrak{T}  $ is formed of planar decorated trees.
\end{itemize} 

Below, we provide an example of such decorated trees:
\begin{equs}
	\CI_a(\Xi_{\mathfrak{l}_2}) X_i  \Xi_{\mathfrak{l}_1}	= 	\begin{tikzpicture}[scale=0.2,baseline=0.1cm]
		\node at (0,0)  [dot,label= {[label distance=-0.2em]below: \scriptsize  $     $} ] (root) {};
		\node at (0,5)  [dot,label= {[label distance=-0.2em]above: \scriptsize  $     $} ] (center) {};
		\node at (-3,4)  [dot,label={[label distance=-0.2em]above: \scriptsize  $  $}] (right) {};
		\node at (-3,8)  [dot,label={[label distance=-0.2em]above: \scriptsize  $  $}] (rightc) {};
		\node at (3,4)  [dot,label={[label distance=-0.2em]above: \scriptsize  $ $} ] (left) {};
		\draw[kernel1] (rightc) to
		node [sloped,below] {\small }     (right);
		\draw[kernel1] (right) to
		node [sloped,below] {\small }     (root); 
		\draw[kernel1] (center) to
		node [sloped,below] {\small }     (root); 
		\draw[kernel1] (left) to
		node [sloped,below] {\small }     (root);
		\node at (2,2) [fill=white,label={[label distance=0em]center: \scriptsize  $ \Xi_{\mathfrak{l}_1} $} ] () {};
		\node at (-3,6) [fill=white,label={[label distance=0em]center: \scriptsize  $ \Xi_{\mathfrak{l}_2} $} ] () {};
		\node at (-2,2) [fill=white,label={[label distance=0em]center: \scriptsize  $ a $} ] () {};
		\node at (0,2.5) [fill=white,label={[label distance=0em]center: \scriptsize  $ X_i $} ] () {};
	\end{tikzpicture} \neq 	\begin{tikzpicture}[scale=0.2,baseline=0.1cm]
		\node at (0,0)  [dot,label= {[label distance=-0.2em]below: \scriptsize  $     $} ] (root) {};
		\node at (0,5)  [dot,label= {[label distance=-0.2em]above: \scriptsize  $     $} ] (center) {};
		\node at (3,4)  [dot,label={[label distance=-0.2em]above: \scriptsize  $  $}] (right) {};
		\node at (0,9)  [dot,label={[label distance=-0.2em]above: \scriptsize  $  $}] (rightc) {};
		\node at (-3,4)  [dot,label={[label distance=-0.2em]above: \scriptsize  $ $} ] (left) {};
		\draw[kernel1] (right) to
		node [sloped,below] {\small }     (root); 
		\draw[kernel1] (center) to
		node [sloped,below] {\small }     (rightc); 
		\draw[kernel1] (center) to
		node [sloped,below] {\small }     (root); 
		\draw[kernel1] (left) to
		node [sloped,below] {\small }     (root);
		\node at (-2,2) [fill=white,label={[label distance=0em]center: \scriptsize  $ X_i$} ] () {};
		\node at (2,2) [fill=white,label={[label distance=0em]center: \scriptsize  $\Xi_{\mathfrak{l}_1}  $} ] () {};
		\node at (0,7) [fill=white,label={[label distance=0em]center: \scriptsize  $ \Xi_{\mathfrak{l}_2} $} ] () {};
		\node at (0,2.5) [fill=white,label={[label distance=0em]center: \scriptsize  $ a $} ] () {};
	\end{tikzpicture} =   X_i \CI_a(\Xi_{\mathfrak{l}_2}) \Xi_{\mathfrak{l}_1}.
\end{equs}

We quotient these decorated trees by the following relations:
\begin{equation}  \label{relation_trees}
	\begin{aligned}
		X_i X_j & = X_j X_i, \quad   \I_a(\tau) \I_{b}(\sigma) =\I_{b}(\sigma) \I_a(\tau)
		\\ \I_a(\tau)X_i & = X_i \I_a(\tau) +  \I_{a-e_i}(\tau),
	\end{aligned}
\end{equation}
We denote by $ \CT $ the linear span of $ \mathfrak{T} $ quotiented by the relations described above.
Let us recall how such decorated trees are used for coding a Butcher series type expansion of the solution of \eqref{set01} within the framework of the theory of regularity structures. This expansion is local around a space-time point $ z $:
\begin{equs}
	U = \sum_{k} \Pi_z(X^k) \frac{u_{k}(z)}{k!} + \sum_{\CI_0(\tau) \in \CP\mathfrak{T}} \frac{\Pi_z(\CI_{0}(\tau))}{S(\tau)} F(\tau)({\bf{u}}(z)).
\end{equs}
where $ u_k =\partial_{x_0}^{k_0}\cdots \partial_{x_d}^{k_d} u $, $ \CP\mathfrak{T} $ are planted trees of the form $ \CI_0(\tau) $ with $ \tau $ a tree in $ \mathfrak{T} $ such that on each node the  $ X_i $ are located at the left-most location.
For a decorated tree
\begin{equs}
  \tau = X^k  \prod_{i=1}^n \CI_{a_i}(\tau_i) \Xi_{\mathfrak{l}},
\end{equs}
the symmetry factor, denoted by $ S(\tau) $, is defined inductively by
\begin{equs} \label{symmetry_factor}
  S(\tau) = (k!) \prod_{j=1}^{r} (m_j !) S(\tau_j)^{m_j},
\end{equs}
where $m_j$ are multiplicities of distinct $\CI_{a_j}(\tau_j)$ in $\prod_{i=1}^n \CI_{a_i}(\tau_i)$, and the elementary differentials 
 $ F(\tau) $ are defined inductively by
\begin{equation} \label{def_upsilon}
	\begin{aligned}
		F(\Xi_l) = a^{\mathfrak{l}}, \qquad F(\tau) = \Big\{\partial^k \partial_{u_{a_1}} ... \partial_{u_{a_n}} F (\Xi_{\mathfrak{l}})\Big\}\,\prod_{j=1}^n F(\tau_j).
\end{aligned} \end{equation}
where $ \partial^k = \partial_{x_0}^{k_0}\cdots \partial_{x_d}^{k_d} $ for $ k \in \mathbb{N}^{d+1}$.
The map $ \Pi_z $ is recursively defined by
\begin{equs}
	\label{def_Pi}
	\begin{aligned}
	(\Pi_z X^k)(z')  & = (z'-z)^k, \quad (\Pi_z 	\Xi_{\mathfrak{l}} )(z') = \xi_{\mathfrak{l}}(z'),
	\\ (\Pi_z \CI_{a}(\tau))(z')  & = (D^{a} K * \Pi_z \tau)(z') \\ & - \sum_{|k|_{\s} \leq \deg(\CI_{a}(\tau))} \frac{(z'-z)^k}{k!} (D^{a +k} K * \Pi_z \tau)(z)
	\end{aligned}
\end{equs}
and extended multiplicatively for the tree product. By tree product, we mean the merging root product $ \cdot $ given by:
\begin{equs}
	X^k	\prod_{i} \CI_{a_i}(\sigma_i) \cdot X^m	\prod_{j} \CI_{b_j}(\sigma_j) = X^{k+m} \prod_{i} \CI_{a_i}(\sigma_i) \	\prod_{j} \CI_{b_j}(\sigma_j)
\end{equs}
where now the products $ \prod_i $ and $ \prod_j $ are commutative. The sum over $k$ in \eqref{def_Pi} is truncated depending on the number $ \deg(\CI_{a}(\tau)) $ which is the degree of $ \CI_{a}(\tau)$. It is computed from the regularity of the noise, the Schauder estimates given by the convolution via the kernels $ D^a K $ and the monomials $ X^k $.  We start by defining  $ |k|_{\s} $ which is given by $ |k|_{\s} = \sum_{i=0}^d \s_i k_i $ where $ \s \in \mathbb{N}^{d+1} $ is a scaling ($(2,1,..,1)$ for parabolic equations as time counts double in comparison to spatial components).
Then, we suppose given $\deg(\Xi_{\mathfrak{l}})$ (meaning that $\xi_{\mathfrak{l}} \in \mathcal{C}^{\deg(\Xi_{\mathfrak{l}})}$, the space-time Hölder regularity of the noise $ \xi_{\mathfrak{l}} $) and $ \deg(\mathcal{I}) $ the Schauder estimate (meaning that $f \in \mathcal{C}^{\alpha} \mapsto K * f \in \mathcal{C}^{\alpha + \deg(\mathcal{I})}$) one has
\begin{equs}
	\deg(\tau \sigma) & = \deg(\tau) + \deg(\sigma),
	\\ 
	\deg(\mathcal{I}_a(\tau)) & = \deg(\mathcal{I}) - |a|_{\s} + \deg(\tau), \quad \deg(X^k) = |k|_{\s}.
\end{equs} 
	
 Let us mention that the map $ \Pi_z $ can be obtained from the map $ \Pi $ which are the iterated integrals without the recentering, also called pre-model:
\begin{equs} \label{def_Pi_z}
	(\Pi X^k)(z')  & = (z')^k, \quad (\Pi 	\Xi_{\mathfrak{l}} )(z') = \xi_{\mathfrak{l}}(z'),
	\\ (\Pi \CI_{a}(\tau))(z')  & = (D^{a} K * \Pi \tau)(z') 
\end{equs}
and extended multiplicatively for the tree product.
Then there exists a linear map $ F_z $ on decorated trees such that
\begin{equs}
	\Pi_z = \Pi \circ F_z.
\end{equs}
One can notice that by our convention, one can associate decorated trees that do have the $ X_i $ on the left-most location to recentered iterated integrals via the map $ \Pi_z $.
The other decorated trees that have a different order on the $ X_i $ could be seen as sum of these iterated integrals once they have been rewritten in the correct basis. For example, let us consider $ \tau $ given by:
\begin{equs} \label{example_1}
	\tau = \mathcal{I}_{b}(\Xi_{\mathfrak{l}}) X_i \Xi_{\mathfrak{l}}
\end{equs}
where $ \mathfrak{l} \in \Lab^- $.
Then, using the relation \eqref{relation_trees}, one has
\begin{equs}
	\tau = X_i \mathcal{I}_{b}(\Xi_{\mathfrak{l}})  \Xi_{\mathfrak{l}} + \mathcal{I}_{b-e_i}(\Xi_{\mathfrak{l}})  \Xi_{\mathfrak{l}}
\end{equs}
and by applying $ \Pi $, one gets
\begin{equs}
	(\Pi \tau)(z') & =  	(\Pi X_i \mathcal{I}_{b}(\Xi_{\mathfrak{l}}) \Xi_{\mathfrak{l}})(z')  + 	(\Pi \mathcal{I}_{b-e_i}(\Xi_{\mathfrak{l}}) \Xi_{\mathfrak{l}})(z') 
	\\ & = z'_i (D^{b} K * \xi_{\mathfrak{l}})(z') (\xi_{\mathfrak{l}})(z') + (D^{b-e_i} K * \xi_{\mathfrak{l}})(z') (\xi_{\mathfrak{l}})(z').
\end{equs}
At the level of the elementary differentials, the order at each node corresponds to the order of derivatives. Indeed, one can naturally extend the definition \eqref{def_upsilon} in order to accommodate these extra terms. For $ \tau =  (\prod_{i=1}^n A_i) \Xi_{\mathfrak{l}} $, one has 
\begin{equs}
	\label{def_upsilon_general}
	\begin{aligned}
		F(\Xi_l) = a^{\mathfrak{l}}, \qquad F(\tau) = \Big\{\prod_{i=1}^n D_{A_i} F (\Xi_{\mathfrak{l}})\Big\}\,\prod_{j=1}^n F(\tau_j).
	\end{aligned}
\end{equs}
where if $ A_i = X_j $, then  $ D_{A_i} = \partial_j $ alse if $ A_i =  \CI_{a_i}(\tau_i) $ then $ D_{A_i} = \partial_{u_{a_i}} $. Continuing the example of $ \tau $ given in \eqref{example_1}, one has
\begin{equs}
	F(\tau) & = F( X_i \mathcal{I}_{b}(\Xi_{\mathfrak{l}})  \Xi_{\mathfrak{l}}) + F(\mathcal{I}_{b-e_i}(\Xi_{\mathfrak{l}})  \Xi_{\mathfrak{l}})
	\\ & = a^{\mathfrak{l}}  \partial_i \partial_{u_b} a^{\mathfrak{l}} + a^{\mathfrak{l}}   \partial_{u_{b-e_i}} a^{\mathfrak{l}}
	\\ & = a^{\mathfrak{l}}   \partial_{u_b} \partial_i a^{\mathfrak{l}} 
\end{equs}
in the last line, we recognise the extension proposed in \eqref{def_upsilon_general}.

In \cite{BCCH}, the vector space $\CT$ was given a multi-pre-Lie algebra structure with the grafting products $\triangleright_\bn$, $\bn\in\mathbb{N}^{d+1}$. The following result is established in \cite{BCCH,BM22}.

\begin{theorem} \label{main_theorem_trees} The multi-pre-Lie algebra $ \CT $ is isomorphic to the free pre-Lie algebra generated by all elements  $ X^k \Xi_{\mathfrak{l}} $.
\end{theorem}

Note that the last of Relations \eqref{relation_trees} can be rewritten as
 \[
(-X_i) \frac{\I_a(\tau)}{a!}=\frac{\I_a(\tau)}{a!}(-X_i) + a_i\frac{\I_{a-e_i}(\tau)}{(a-e_i)!},
 \]
which easily implies that, if we let $\partial_i=-X_i$ and create combinatorial factor $\frac{1}{a!}$ for each edge labelled $a$ of each rooted tree, the vector space $\CT$ is the free extended multi-pre-Lie algebra. Using Proposition \ref{prop:extlin}, we immediately obtain another proof of Theorem \ref{main_theorem_trees} that is completely parallel to that of Theorem \ref{second_main_theorem_multi_indices}.

Let us explore this observation a bit further. Since the multi-Novikov identities imply the multi-pre-Lie identities, the free multi-Novikov algebra $\Nov_A\langle X\rangle$ can also be constructed as a quotient of the free multi-pre-Lie algebra $\PL_A\langle X\rangle$. Important elements of free multi-pre-Lie algebras are the following ``multi-symmetric braces'' that generalise the usual symmetric braces \cite{LM}.

\begin{definition} \label{multi_braces}
Let $a_1,\ldots,a_n\in A$. For elements $x,y_1,\ldots,y_n$ of a multi-pre-Lie algebra $V$, their \emph{multi-symmetric brace} 
 \[
\{x;y_1,\ldots,y_n\}_{a_1,\ldots,a_n}
 \] 
is defined by the recursive rule $\{x;y\}_a=x\triangleright_a y$ and
\begin{multline*}
\{x;y_1,\ldots,y_{n+1}\}_{a_1,\ldots,a_{n+1}}=\\ 
y_{n+1}\triangleright_{a_{n+1}}\{x;y_1,\ldots,y_{n}\}_{a_1,\ldots,a_{n}} -\\ \sum_{i=1}^n
\{x;y_1,\ldots,y_{n+1}\triangleright_{a_{n+1}}y_i,\ldots,y_{n}\}_{a_1,\ldots,a_{n}}.
\end{multline*}
\end{definition}

It is shown in \cite{BCCH} that, similarly to the classical construction of Chapoton and Livernet \cite{ChaLiv}, one can construct the algebra $\PL_A\langle X\rangle$ as the vector space with the basis $\RT_A\langle X\rangle$, the set of all rooted trees whose vertices are labelled by elements of $X$ and whose edges are labelled by elements of $A$. The key building blocks of which rooted trees are constructed are \emph{corollas}, trees with only one internal vertex and several leaves. One can prove by a simple inductive argument that whenever $x,y_1\ldots,y_n\in X$, the multi-symmetric brace $\{x;y_1,\ldots,y_n\}_{a_1,\ldots,a_n}$ is the corolla with $n$ leaves whose root vertex is labelled $x$, whose leaves are labelled $y_1,\ldots,y_n$, and whose edge incident to the leaf labelled $y_i$ is labelled $a_i$. 

It turns out that, for the canonical multi-Novikov structure on the free commutative multidifferential algebra (and hence for  the multidifferential realisation of the free Novikov algebra), the multi-symmetric braces are given by a particularly simple formula, proved by a simple inductive argument.
\begin{proposition} \label{brace_Novikov}
For the canonical multi-Novikov structure on $\CDi_A\langle X\rangle$, we have 
 \[
\{x;y_1,\ldots,y_{n+1}\}_{a_1,\ldots,a_{n+1}}=\partial_{a_1}\cdots\partial_{a_n+1}(x)y_1\cdots y_{n+1}.
 \]
\end{proposition}

Thus, the quotient morphism from the free multi-pre-Lie algebra to the free multi-Novikov algebra ``disassembles'' a tree by applying derivatives of incoming edges of each vertex to the label of that vertex. This recovers in a conceptual way the morphism from decorated trees to multi-indices of \cite[Remark 5.6]{BK23}.

\end{document}